\documentclass[12pt,reqno]{amsart}
\usepackage{amsfonts}

\usepackage{amssymb,amsmath,graphicx,amsfonts,euscript}
\usepackage{color}

\newtheorem{thm}{Theorem}[section]

\newtheorem{lemma}{Lemma}[section]

\theoremstyle{definition}
\newtheorem{define}{Definition}[section]

\theoremstyle{remark}
\newtheorem{rem}{Remark}[section]

\allowdisplaybreaks \voffset=-0.2in \numberwithin{equation}{section}

\begin{document}
\bigskip

\centerline{\Large\bf  Global regularity of the two-dimensional regularized }
\smallskip
\centerline{\Large\bf MHD equations}
\bigskip

\centerline{Zhuan Ye}

\medskip

\centerline{Department of Mathematics and Statistics, Jiangsu Normal University,
}
\medskip

\centerline{101 Shanghai Road, Xuzhou 221116, Jiangsu, PR China}

\medskip

\centerline{E-mail: \texttt{yezhuan815@126.com
}}

\bigskip

{\bf Abstract:}~~%
In this paper, we consider the Cauchy problem of the two-dimensional regularized incompressible magnetohydrodynamics equations. The main objective of this paper is to establish the global regularity of classical solutions of the magnetohydrodynamics equations with the minimal dissipation. Consequently, our results significantly improve the previous works.

{\vskip 1mm
 {\bf AMS Subject Classification 2010:}\quad 35Q35; 35B65; 76D03.

 {\bf Keywords:}
MHD equations; Global regularity; Fractional dissipation.}

\vskip .4in
\section{Introduction and main results}

Whether or not the smooth solutions of the classical two-dimensional (2D) incompressible magnetohydrodynamics (MHD) equations with only velocity dissipation or with only magnetic diffusion develop finite time singularities is an extremely difficult problem and remains open. The main difficulty arising here is a strong impact of the higher modes to the leading order dynamics through the nonlinearity which a priori may destroy the regularity of a solution and thus lead to the formation of singularities. To bypass this difficulty, considerable models have been proposed in order to capture the leading dynamics of the flow on the one hand and somehow suppress the higher modes on the other hand, see for example \cite{Cheskidovhot,Foiasht,Ilyintt,HO,MininniMp,LinshizT,Zhoufan,Zhoufanna11}.
In this paper, we are interested in studying the Cauchy
problem of the following 2D regularized MHD equations
\begin{equation}\label{RG2DMHD}
\left\{\aligned
&\partial_{t}v+(u\cdot\nabla)v+(-\Delta)^{\alpha} v+\sum_{j=1}^{2}v_{j}\nabla u_{j}+\nabla \Big(p+\frac{1}{2} |b|^{2}\Big)=(b\cdot\nabla) b, \\
&\partial_{t}b+(u\cdot\nabla)b+(-\Delta)^{\beta} b=(b\cdot\nabla)u,\\
&v=u+(-\Delta)^{\gamma}u,\\
&\nabla\cdot u=\nabla\cdot v=\nabla\cdot b=0,\\
&v(x,0)=v_{0}(x),\ \ b(x,0)=b_{0}(x),\quad x\in \mathbb{R}^{2},
\endaligned \right.
\end{equation}
where $v=(v_{1},\,v_{2})$ denotes the velocity vector, $b=(b_{1},\,b_{2})$ the magnetic field, $u=(u_{1},\,u_{2})$ the "filtered" velocity and $p$ the scalar "filtered" pressure, respectively, $\alpha\in[0,\,2]$, $\beta\in[0,\,2]$ and $\gamma\in[0,\,2]$ are real parameters. $v_{0}(x)$ and $b_{0}(x)$ are the given initial data satisfying $\nabla\cdot v_{0}=\nabla\cdot b_{0}=0$.
The fractional Laplacian operator $(-\Delta)^{\sigma}$ is defined through
the Fourier transform, namely
$$\widehat{(-\Delta)^{\sigma} f}(\xi)=|\xi|^{2\sigma}\hat{f}(\xi),$$
where
$$\hat{f}(\xi)=\frac{1}{(2\pi)^{2}}\int_{\mathbb{R}^{2}}e^{-ix\cdot\xi}f(x)\,dx.$$
For simplicity, we use the notation $\Lambda:=(-\Delta)^{\frac{1}{2}}$.
We remark the convention that by $\alpha=0$ we mean that there is no dissipation in $\eqref{RG2DMHD}_{1}$, and similarly $\beta=0$ represents that there is no dissipation in $\eqref{RG2DMHD}_{2}$. Moreover, $\gamma=0$ represents that the system $\eqref{RG2DMHD}$ without regularizing, namely, $v=2u$.
Recently in the literature great attention has been devoted to the study of fractional Laplace-type problems, not only for pure mathematical generalization, but also for various applications in different fields.
In fact, the application background of fractional problems can be founded in fractional quantum mechanics \cite{Laskin}, probability \cite{Applebaum}, overdriven detonations in gases \cite{Clavind}, anomalous diffusion in semiconductor growth \cite{Woyczyh}, physics and chemistry \cite{Metzlerk} and so on.

\vskip .1in
Let us briefly discuss the rich history concerning the system \eqref{RG2DMHD}.
When $\gamma=0$, \eqref{RG2DMHD} reduces to the 2D generalized magnetohydrodynamics (GMHD) system, which describes the motion of electrically conducting fluids and has broad applications in applied sciences such as astrophysics, geophysics and plasma physics (see, e.g., \cite{Davidson}).
Besides their wide physical applicability, the GMHD equations are also of great interest in mathematics.
As a coupled system, the GMHD equations contain much richer structures than
the Navier-Stokes-type equations. The GMHD equations are not merely a combination of two parallel Navier-Stokes type equations but an interactive and integrated system. These distinctive features make analytic studies a great challenge but offer new opportunities.
Since the work \cite{Wujde3}, the mathematical analysis (specially the global regularity problem) of the 2D GMHD equations
has attracted considerable interests in the last few years and progress has been made (see \cite{CWY,Cailei,CW1,FNZ,JZ01,Leizhou,LINXUZ,TYZ1,WEIZHANG,WuWuXu,
WuJMFM,Yedfsdf8,Yamzaki14aml,TrYuna,YuanZhao} with no intention to be complete). We remark that the latest global regularity results of the 2D GMHD equations can be summarized as
$$(1)\ \ \alpha>0,\,\,\beta=1;\quad (2)\ \ \alpha=0,\,\,\beta>1;\quad (3)\ \ \alpha=2,\,\,\beta=0;\quad (4)\ \ \alpha+\beta=2,\,\,\beta<1,$$
see \cite{CWY,FNZ,JZ01} for details (one also refers to \cite{Agelas,Yedfsdf8,YuanZhao} for logarithmic type result).

\vskip .1in
Now we would like to review some very related works about the system \eqref{RG2DMHD} with $\gamma=1$, which is the so-called classical regularized GMHD equations (so-called Leray-alpha model). This system was obtained from variational principles by
modifying the Hamiltonian associated with the standard MHD equations
(see \cite{HO}).
Linshiz-Titi \cite{LinshizT} introduced the system \eqref{RG2DMHD} with $\alpha=\beta=\gamma=1$ and proved that the the corresponding system admits a unique global smooth solution both in 2D case and 3D case. Later, Fan and Ozawa \cite{FO} established the global existence of smooth solution for the classical 2D regularized GMHD equations with $\alpha=0,\,\beta=1$ or $\alpha=1,\,\beta=0$, which was further extended logarithmically by \cite{KCYakijmaa}.
Recently, Zhao and Zhu \cite{ZZhu} proved the global regularity result for another special case $\alpha=\beta=\frac{1}{2}$ in 2D case. Very recently, Yamazaki \cite{Yamazaki16} further examined the general case, namely, $\alpha+\beta=1$ with $\alpha,\,\beta\in(0,\,1)$ and obtained the global regularity for this general case. The author \cite{YeXuaml} considered the 3D case and proved the global regularity under the assumption $\alpha+\beta\geq\frac{3}{2}$ and $\frac{3}{4}\leq\alpha<\frac{3}{2}$. For the endpoint case $\alpha=\frac{3}{2},\,\beta=0$, we refer to \cite{xuyecpaa,Yamaz16jmaa}.

\vskip .1in
The aim of this paper is to gain further understanding of the global regularity problem for the 2D regularized system (\ref{RG2DMHD}) with only velocity dissipation or with only magnetic diffusion. We present two main results in
hope that they may shed light on the eventual resolution of the global regularity problem of the classical 2D incompressible MHD equations with only velocity dissipation or with only magnetic diffusion, namely \eqref{RG2DMHD} with $\alpha=1,\,\beta=\gamma=0$ or $\beta=1,\,\alpha=\gamma=0$ .
More precisely, the first result concerns (\ref{RG2DMHD}) with no velocity dissipation, which can be stated as follows.

\vskip .1in
\begin{thm}\label{Th1}
Consider (\ref{RG2DMHD}) with $\alpha=0$, namely,
\begin{equation}\label{eqrtyett01}
\left\{\aligned
&\partial_{t}v+(u\cdot\nabla)v+\sum_{j=1}^{2}v_{j}\nabla u_{j}+\nabla \Big(p+\frac{1}{2} |b|^{2}\Big)=(b\cdot\nabla) b, \\
&\partial_{t}b+(u\cdot\nabla)b+(-\Delta)^{\beta} b=(b\cdot\nabla)u,\\
&v=u+(-\Delta)^{\gamma}u,\\
&\nabla\cdot u=\nabla\cdot v=\nabla\cdot b=0,\\
&v(x,0)=v_{0}(x),\ \ b(x,0)=b_{0}(x),\quad x\in \mathbb{R}^{2}.
\endaligned \right.
\end{equation}
Assume that $v_{0}\in H^{\varrho}(\mathbb{R}^{2}),\,b_{0}\in H^{\rho+1-\beta}(\mathbb{R}^{2})$ with $\rho>\max\{2,\,1+\beta\}$ and $\nabla\cdot v_{0}=\nabla\cdot b_{0}=0$. If $\beta$ and $\gamma$ satisfy
\begin{align}\label{condyet11}
\beta>1-\frac{\gamma}{2}\ \  \mbox{\rm with}\  \gamma\in[0,\,2],
\end{align}
then the system (\ref{eqrtyett01}) admits a unique global regular solution $(v,\,b)$ such that for any given $T>0$,
$$v\in L^{\infty}([0, T]; H^{\rho}(\mathbb{R}^{2})),\quad  b \in L^{\infty}([0, T]; H^{\rho+1-\beta}(\mathbb{R}^{2}))\cap L^{2}([0, T]; H^{\rho+1}(\mathbb{R}^{2})).
$$
\end{thm}

\begin{rem}  \rm
As stated above, Fan and Ozawa \cite{FO} established the global regularity for the system (\ref{RG2DMHD}) with $\alpha=0,\,\beta=\gamma=1$, which was further extended logarithmically by \cite{KCYakijmaa}. We remark that under the assumption $\alpha=0,\,\gamma=1$, the requirement $\beta>\frac{1}{2}$ actually would ensure the global regularity. Consequently, Theorem \ref{Th1} significantly improves the previous works \cite{FO,KCYakijmaa,Yamazaki16,ZZhu} and answers an interesting problem remarked in \cite[Remark 1.2]{Yamazaki16}.
\end{rem}

\vskip .1in
The following theorem concerns (\ref{RG2DMHD}) with no magnetic diffusion.

\begin{thm}\label{Th2}
Consider (\ref{RG2DMHD}) with $\beta=0$ and the logarithmic dissipation in velocity equation, namely,
\begin{equation}\label{logRG2DMHD}
\left\{\aligned
&\partial_{t}v+(u\cdot\nabla)v+(-\Delta)^{\alpha}\mathcal{L}^{2} v+\sum_{j=1}^{2}v_{j}\nabla u_{j}+\nabla \Big(p+\frac{1}{2} |b|^{2}\Big)=(b\cdot\nabla) b, \\
&\partial_{t}b+(u\cdot\nabla)b=(b\cdot\nabla)u,\\
&v=u+(-\Delta)^{\gamma}\mathcal{L}^{2}u,\\
&\nabla\cdot u=\nabla\cdot v=\nabla\cdot b=0,\\
&v(x,0)=v_{0}(x),\ \ b(x,0)=b_{0}(x),\quad x\in \mathbb{R}^{2},
\endaligned \right.
\end{equation}
where the operator $\mathcal{L}$ is defined by
$$\widehat{\mathcal{L}v}(\xi)=\frac{1}{g(\xi)}\widehat{v}(\xi)$$
for some non-decreasing symmetric function $g(\tau)\geq 1$ defined on $\tau\geq 0$. Let $v_{0}\in H^{\varrho}(\mathbb{R}^{2}),\,b_{0}\in H^{\rho+1-\alpha'}(\mathbb{R}^{2})$ with $\alpha>\alpha'$ and $\rho>\max\{2,\,1+\alpha'\}$ as well as $\nabla\cdot v_{0}=\nabla\cdot b_{0}=0$. If $\alpha+\gamma=2$ with $\alpha\in(0,\,2]$ and $g$ satisfies the following growth condition
\begin{equation}\label{logcobd}
\int_{e}^{\infty}\frac{d\tau}{\tau\sqrt{\ln \tau} g^{2}(\tau)}=\infty,\end{equation}
then the system (\ref{logRG2DMHD}) admits a unique global regular solution $(v,\,b)$ such that for any given $T>0$
$$v\in L^{\infty}([0, T]; H^{\rho}(\mathbb{R}^{2}))\cap L^{2}([0, T]; H^{\rho+\alpha'}(\mathbb{R}^{2})),\quad  b \in L^{\infty}([0, T]; H^{\rho+1-\alpha'}(\mathbb{R}^{2})).
$$
\end{thm}

\begin{rem}  \rm
We point out that the global regularity of the system (\ref{RG2DMHD}) with $\beta=0,\,\alpha=\gamma=1$ was proved by Fan and Ozawa \cite{FO}, which was further extended logarithmically by \cite[Theorem 1.1]{KCYakijmaa}.
Consequently, Theorem \ref{Th2} improves the results of \cite[Theorem 1.1]{FO} and \cite[Theorem 1.1]{KCYakijmaa}.
\end{rem}

\begin{rem}\rm
We should remark that the typical examples satisfying the condition (\ref{logcobd}) are as follows
\begin{equation}
\begin{split}
& g(\xi)=\big[\ln(e+|\xi|)\big]^{\frac{1}{4}}; \\
& g(\xi)=\big[\ln(e+|\xi|)\big]^{\frac{1}{4}}\big[\ln (e+\ln(e+|\xi|)) \big]^{\frac{1}{2}};\\
& g(\xi)=\big[\ln(e+|\xi|)\big]^{\frac{1}{4}}\big[\ln (e+\ln(e+|\xi|)) \ln(e+\ln(e+\ln(e+|\xi|)))\big]^{\frac{1}{2}}.\nonumber
\end{split}
\end{equation}
\end{rem}

\vskip .1in
The next main result concerned with the endpoint case, namely, (\ref{RG2DMHD}) with $\alpha=\beta=0$ and $\gamma=2$, can be stated as follows.
\begin{thm}\label{addTh2}
Consider (\ref{RG2DMHD}) with $\alpha=\beta=0$ and $\gamma=2$, namely
\begin{equation}\label{endlogRG2DMHD}
\left\{\aligned
&\partial_{t}v+(u\cdot\nabla)v+\sum_{j=1}^{2}v_{j}\nabla u_{j}+\nabla \Big(p+\frac{1}{2} |b|^{2}\Big)=(b\cdot\nabla) b, \\
&\partial_{t}b+(u\cdot\nabla)b=(b\cdot\nabla)u,\\
&v=u+(-\Delta)^{2}\mathcal{L}^{2}u,\\
&\nabla\cdot u=\nabla\cdot v=\nabla\cdot b=0,\\
&v(x,0)=v_{0}(x),\ \ b(x,0)=b_{0}(x),\quad x\in \mathbb{R}^{2},
\endaligned \right.
\end{equation}
where the operator $\mathcal{L}$ is defined by
$$\widehat{\mathcal{L}v}(\xi)=\frac{1}{g(\xi)}\widehat{v}(\xi)$$
for some non-decreasing symmetric function $g(\tau)\geq 1$ defined on $\tau\geq 0$. Assume that $v_{0}\in H^{\varrho}(\mathbb{R}^{2}),\,b_{0}\in H^{\rho+1}(\mathbb{R}^{2})$ with $\rho>2$ and $\nabla\cdot v_{0}=\nabla\cdot b_{0}=0$. If $g$ satisfies the following growth condition
\begin{equation}\label{endlogcobd}
\int_{e}^{\infty}\frac{d\tau}{\tau\sqrt{\ln \tau} g(\tau)}=\infty,\end{equation}
then the corresponding system (\ref{endlogRG2DMHD}) admits a unique global regular solution $(v,\,b)$ such that for any given $T>0$,
$$v\in L^{\infty}([0, T]; H^{\rho}(\mathbb{R}^{2})),\quad  b \in L^{\infty}([0, T]; H^{\rho+1}(\mathbb{R}^{2})).
$$
\end{thm}

\begin{rem}\rm
We also remark that the typical examples satisfying the condition (\ref{endlogcobd}) are
\begin{equation}
\begin{split}
& g(\xi)=\big[\ln(e+|\xi|)\big]^{\frac{1}{2}}; \\
& g(\xi)=\big[\ln(e+|\xi|)\big]^{\frac{1}{2}} \ln (e+\ln(e+|\xi|));\\
& g(\xi)=\big[\ln(e+|\xi|)\big]^{\frac{1}{2}} \ln (e+\ln(e+|\xi|)) \ln(e+\ln(e+\ln(e+|\xi|))).\nonumber
\end{split}
\end{equation}
\end{rem}

\begin{rem}  \rm
The requirements $\rho>\max\{2,\,1+\beta\}$ in Theorem \ref{Th1} and $\rho>\max\{2,\,1+\alpha\}$ in Theorem \ref{Th2} guarantee the initial data belonging to the lipschitzian class. As a matter of fact, this requirements can be weakened. In order to avoid the tedious computations, we just work on these functional spaces.
\end{rem}

\begin{rem}  \rm
Using the techniques and arguments of this paper, it is not hard to derive the global regularity of the system (\ref{RG2DMHD}) with $\alpha+\beta+\gamma=2$. We would like to leave it to the interested readers. However, it seems to be an interesting problem whether or not the global regularity result of the system (\ref{RG2DMHD}) with $\alpha+\beta+\gamma<2$ and $\beta\leq 1-\frac{\gamma}{2}$.
\end{rem}

Finally, inspired by the proof of Theorem \ref{Th1}, we are able to show the global regularity for the following 2D regularized MHD equations \cite{LinshizT}
\begin{equation}\label{addRG2DMHD}
\left\{\aligned
&\partial_{t}v+(u\cdot\nabla)v+\nabla \Big(p+\frac{1}{2} |b|^{2}\Big)=(b\cdot\nabla) b, \\
&\partial_{t}b+(u\cdot\nabla)b+(-\Delta)^{\beta} b=(b\cdot\nabla)u,\\
&v=u-\Delta u,\\
&\nabla\cdot u=\nabla\cdot v=\nabla\cdot b=0,\\
&v(x,0)=v_{0}(x),\ \ b(x,0)=b_{0}(x),\quad x\in \mathbb{R}^{2}.
\endaligned \right.
\end{equation}
More precisely, we have the following result.
\begin{thm}\label{Th3}
Assume that $v_{0}\in H^{\varrho}(\mathbb{R}^{2}),\,b_{0}\in H^{\rho+1-\beta}(\mathbb{R}^{2})$ with $\rho>\max\{2,\,1+\beta\}$ and $\nabla\cdot v_{0}=\nabla\cdot b_{0}=0$. If $\beta>\frac{1}{2}$, then the system \eqref{addRG2DMHD} admits a unique global regular solution $(v,\,b)$ such that for any given $T>0$,
$$v\in L^{\infty}([0, T]; H^{\rho}(\mathbb{R}^{2})),\quad  b \in L^{\infty}([0, T]; H^{\rho+1-\beta}(\mathbb{R}^{2}))\cap L^{2}([0, T]; H^{\rho+1}(\mathbb{R}^{2})).
$$
\end{thm}

\begin{rem}  \rm
The global regularity result of the system \eqref{addRG2DMHD} with $\beta=1$ was established in \cite[Theorem 1.2]{Zhoufan}. Very recently, this global regularity result in the logarithmically supercritical regime was extended by KC and Yamazaki in \cite[Theorem 1.1]{KCYamazaki}. Clearly, Theorem \ref{Th3} can be regarded as a significant improvement of \cite[Theorem 1.2]{Zhoufan} and \cite[Theorem 1.1]{KCYamazaki}. As the proof of Theorem \ref{Th3} can be performed via the similar arguments adopted in proving Theorem \ref{Th1} with suitable modifications, we only sketch its proof in Appendix \ref{aaapA}.
\end{rem}

\vskip .2in
Now we give some rough ideas on our proof of Theorem \ref{Th1} and Theorem \ref{Th2}. The proof is not straightforward and demands new techniques. We describe
the main difficulties and explain the techniques to overcome them.
The key to the global regularity is the global {\it a priori} bounds.
We begin with  Theorem \ref{Th1}.

\begin{center}
$$\textbf{Case 1}:  \  \beta>1-\frac{\gamma}{2}\ \  \mbox{\rm with}\  \gamma\in(0,\,1] $$
\end{center}
\vskip .1in

The following basic global $L^2$-energy is immediate due to the special structure of \eqref{eqrtyett01}
\begin{align}\label{exstp1}
\|u(t)\|_{H^{\gamma}}^{2}
+\|b(t)\|_{L^{2}}^{2}+\int_{0}^{t}{\|\Lambda^{\beta}b(\tau)\|_{L^{2}}^{2}\,d\tau}\leq
C(v_{0},b_{0}),
\end{align}
which serves as a preparation for higher regularity estimates. The next step is to derive the global $H^1$-bound, but direct energy estimates do not appear to easily
yield this bound. One of the difficulties comes from the nonlinear term in the velocity equation. Moreover, there is no dissipation in the velocity equation and the dissipation in the magnetic equation is not strong enough (only $\beta\leq1$). To overcome these difficulties, we first make use of the bound \eqref{exstp1} to show the following optimal regularity for $b$
\begin{align}\label{exstp2}
\sup_{t\in [0,\,T]}\|\Lambda^{\beta+\gamma-1}b(t)\|_{L^{2}}^{2}
+\int_{0}^{T}{\|\Lambda^{2\beta+\gamma-1}b(\tau)\|_{L^{2}}^{2}\,d\tau}\leq
C(T,v_{0},b_{0}).
\end{align}
Next, if one tries to obtain more regularity for $b$, then the higher regularity for $u$ or $v$ will be needed. To this end, we appeal to the vorticity $\omega:=\nabla^{\perp}\cdot v\equiv \partial_{1} v_{2}-\partial_{2} v_{1}$ equation
\begin{align} \label{vircitye1}
\partial_{t}\omega+(u\cdot\nabla)\omega=\nabla^{\perp}\nabla\cdot(b\otimes b).
\end{align}
More precisely, combining the estimates for $\omega$ and $b$, we derive by using \eqref{exstp2} that
\begin{gather}
\sup_{t\in [0,\,T]}(\|v(t)\|_{H^{1}}^{2}+\|\Lambda^{2\gamma+\beta}b(t)\|_{L^{2}}^{2})
+\int_{0}^{T}{\|\Lambda^{2\gamma+2\beta}b(\tau)\|_{L^{2}}^{2}\,d\tau}\leq
C(T,v_{0},b_{0}),\label{exstp3}\\
\sup_{t\in [0,\,T]}\|b(t)\|_{L^{\infty}}
+\int_{0}^{T}{\|\nabla b(\tau)\|_{L^{\infty}}^{2}\,d\tau}\leq
C(T,v_{0},b_{0}).\label{exstp4}
\end{gather}
At this moment, the estimate for $b$ in \eqref{exstp4} is enough to obtain the higher regularity. Consequently, it suffices to derive more regularity estimate for $u$ or $v$. Due to the nonlinear term in the velocity equation, it is natural to take advantage of the vorticity equation \eqref{vircitye1}. But the obstacle is to control the two-order derivatives of $b$. To achieve this goal, we first rewrite the equation $\eqref{eqrtyett01}_{2}$ due to $\nabla\cdot u=\nabla\cdot b=0$
\begin{align}\label{exstp9}
\partial_{t}b+\Lambda^{2\beta} b=\nabla\cdot(b\otimes u-u\otimes b).
\end{align}
Applying the space-time estimate (see Lemma \ref{afasqw6sf}) to \eqref{exstp9}, it allows us to show the following crucial inequality (see \eqref{RMHDTQye20} for details)
\begin{align}\label{exstp10}
\|\Lambda^{2\beta+2\gamma-\epsilon}b\|_{L_{T}^{1}L^{q}}\leq
C(T,v_{0},b_{0})+C(T,v_{0},b_{0})
\|\omega\|_{L_{T}^{1}L^{q}}.
\end{align}
Unfortunately, at present we have no estimate for $\|\omega\|_{L_{T}^{1}L^{q}}$. In order to overcome this difficulty and to close \eqref{exstp10}, we deduce from the vorticity equation \eqref{vircitye1} that (see \eqref{RMHDTQye23} for details)
\begin{align}\label{exstp11}
\|\omega(t)\|_{L_{T}^{\infty}L^{q}}\leq C(T,v_{0},b_{0})+C(T,v_{0},b_{0})
\|\Lambda^{2\beta+2\gamma-\epsilon}b\|_{L_{T}^{1}L^{q}}
^{\frac{2}{2\beta+2\gamma-\epsilon}},
\end{align}
where $\epsilon\in (0,\,\min\{2\beta-1,\,2\beta+2\gamma-2\})$ and $q\in\big(\frac{1}{2\beta-1-\epsilon},\,\infty\big)$.
Combining \eqref{exstp10} and \eqref{exstp11} yields
\begin{align}\label{exstp12}
\|\Lambda^{2\beta+2\gamma-\epsilon}b\|_{L_{T}^{1}L^{q}}\leq
C(T,v_{0},b_{0}).
\end{align}
As $\beta$ and $\gamma$ satisfy (\ref{condyet11}), we can take $q$ large enough to get from \eqref{exstp12} that
\begin{align*}
\|\nabla^{2}b\|_{L_{T}^{1}L^{\infty}} \leq
C(T,v_{0},b_{0}),
\end{align*}
which along with the vorticity equation \eqref{vircitye1} implies
\begin{align}\label{exstp13}
\sup_{t\in [0,\,T]}\|\omega(t)\|_{L^{\infty}}\leq
C(T,v_{0},b_{0}).
\end{align}
Actually, the above estimate \eqref{exstp13} is a key component to obtain the higher regularity.

 \vskip .1in
\begin{center}
$$\textbf{Case 2}:  \  \beta>1-\frac{\gamma}{2}\ \  \mbox{\rm with}\  \gamma\in(1,\,2]$$
\end{center}
\vskip .1in

We remark that the proof for \textbf{Case 2} can be performed as that of \textbf{Case 1}, but some different techniques and observations are required.
In this case, basic global $L^2$-energy \eqref{exstp1} is still valid.
Moreover, we can also show that the optimal regularity for $b$, namely, \eqref{exstp2} holds true. Our next step is to derive \eqref{exstp3}, namely,
\begin{align}\label{zxdrtgbn89}
\sup_{t\in [0,\,T]}(\|v(t)\|_{H^{1}}^{2}+\|\Lambda^{2\gamma+\beta}b(t)\|_{L^{2}}^{2})
+\int_{0}^{T}{\|\Lambda^{2\gamma+2\beta}b(\tau)\|_{L^{2}}^{2}\,d\tau}\leq
C(T,v_{0},b_{0}).
\end{align}
With the observation $2\gamma+2\beta>3$ due to $\beta>1-\frac{\gamma}{2}$ with $\gamma\in(1,\,2]$, this implies that the regularity for $b$ of \eqref{zxdrtgbn89} is good enough. More precisely, we have at least
 \begin{align*}
\|\nabla^{2}b\|_{L_{T}^{2}L^{\infty}} \leq
C(T,v_{0},b_{0}),
\end{align*}
Consequently, this along with the vorticity equation \eqref{vircitye1} leads to
\begin{align}\label{zqtgbn8ty2}
\sup_{t\in [0,\,T]}\|\omega(t)\|_{L^{\infty}}\leq
C(T,v_{0},b_{0}).
\end{align}

\vskip .1in
Finally, the estimates \eqref{exstp4} and \eqref{exstp13} or (\eqref{zxdrtgbn89} and \eqref{zqtgbn8ty2}) will allow us to propagate all the higher regularities for $(v,\,b)$. This ends the proof of Theorem \ref{Th1}.

\vskip .2in
We now explain the proof of Theorem \ref{Th2}. We first have the following basic global $L^2$-energy
\begin{align}\label{exstp14}
\|u(t)\|_{H^{\gamma}}^{2}
+\|b(t)\|_{L^{2}}^{2}+\int_{0}^{t}{\Big\|\frac{\Lambda^{2}}{ g^{2}(\Lambda)}u \Big\|_{L^{2}}^{2} \,d\tau} \leq
C(v_{0},b_{0}).
\end{align}
The regularity for $u$ in \eqref{exstp14} seems to be good, which is quite important to derive the higher regularity of the solution. Our next target is to show the estimate for $b$. This is not trivial as there is no dissipation in $b$-equation. Actually, by direct computations, we may conclude
\begin{align*}
\frac{d}{dt}V(t)+H(t)\leq C(1+\|\nabla u\|_{L^{\infty}}+\|\Delta u\|_{L^{2}})V(t),
\end{align*}
where
$$V(t):=\|b(t)\|_{L^{\infty}}^{3}+\|b(t)\|_{L^{\infty}}^{6}+
\|\nabla b(t)\|_{L^{2}}^{3}+\|v(t)\|_{L^{2}}^{2},\quad H(t):=\|\mathcal{L}\Lambda^{\alpha}v(t)\|_{L^{2}}^{2}.$$
In order to handle the two terms $\|\nabla u\|_{L^{\infty}}$ and $\|\Delta u\|_{L^{2}}$, we take fully exploit of the Littlewood-Paley technique, which allows us to deduce
\begin{align}
 \frac{d}{dt}V(t)+H(t)
 \leq C\left(1+\Big\|\frac{\Lambda^{2}}{ g^{2}(\Lambda)}u \Big\|_{L^{2}}\right)g^{2}\big[\left(e+V(t)\right)^{\frac{1}{2(r_{1}-2)}}\big]
 \sqrt{\ln\big(e+V(t)\big)} \big(e+V(t)\big).\nonumber
\end{align}
Using \eqref{exstp14} and the condition \eqref{logcobd}, we end up with
$$\sup_{t\in [0,\,T]}V(t)+\int_{0}^{T}H(\tau)\,d\tau\leq C(T,v_{0},b_{0}).$$
However, the regularity for $b$ here is still not enough. To gain higher regularity, we go through the process by improving the regularity for $v$ to improve the regularity for $b$. More precisely, the boundedness of $V(t)$ allows us to show (see Lemma \ref{zxdDSma} for details)
\begin{align}\label{fcftewqq2e}
\sup_{t\in [0,\,T]}\|\Lambda^{1+\nu}b(t)\|_{L^{2}}\leq
C(T,v_{0},b_{0}),\quad \forall\,\nu<\gamma.
\end{align}
The bound \eqref{fcftewqq2e} plays an important role in deriving the following estimate
\begin{align*}
 \sup_{t\in [0,\,T]}\|\Lambda^{\alpha}v(t)\|_{L^{2}}^{2}
+\int_{0}^{T}{\|\Lambda^{2\alpha}\mathcal{L}v(\tau)\|_{L^{2}}^{2} \,d\tau}\leq
C(T,v_{0},b_{0}),
\end{align*}
which further implies (see Lemma \ref{RMHDS33} for details)
\begin{gather}\label{exstp16}
 \sup_{t\in [0,\,T]}\|\nabla b(t)\|_{L^{\infty}}\leq C(T,v_{0},b_{0}).
\end{gather}
Consequently, by \eqref{exstp16}, there holds (see Lemma \ref{RMHDS34} for details)
\begin{align*}
\sup_{t\in [0,\,T]}\|\omega(t)\|_{L^{2}}^{2}
+\int_{0}^{T}{\|\Lambda^{\alpha}\mathcal{L}\omega(\tau)\|_{L^{2}}^{2} \,d\tau}\leq
C(T,v_{0},b_{0}).
\end{align*}
Finally, with the above estimates at our disposal, we can propagate all the higher regularities and thus complete the proof of Theorem \ref{Th2}.

\vskip .1in
The rest of the paper is organized as follows. In Section 2 we carry out the proof of Theorem \ref{Th1}. Section 3 is devoted to the proof of Theorem \ref{Th2} and Theorem \ref{addTh2}.
We sketch the proof of Theorem \ref{Th3} in Appendix \ref{aaapA}, while
in Appendix \ref{aaapB}, we present some useful lemmas of this paper.

\vskip .3in
\section{The proof of Theorem \ref{Th1}}\setcounter{equation}{0}
This section is devoted to the proof of Theorem \ref{Th1}. Before the proof, we will state a notation.
For a quasi-Banach space $X$ and for any $0<T\leq\infty$, we use standard notation $L^{p}(0,T;X)$ or $L_{T}^{p}(X)$ for
the quasi-Banach space of Bochner measurable functions $f$ from $(0,T)$ to $X$ endowed with the norm
\begin{equation*}
\|f\|_{L_{T}^{p}(X)}:=\left\{\aligned
&\left(\int_{0}^{T}{\|f(.,t)\|_{X}^{p}\,dt}\right)^{\frac{1}{p}}, \,\,\,\,\,1\leq p<\infty,\\
&\sup_{0\leq t\leq T}\|f(.,t)\|_{X},\qquad\qquad p=\infty.
\endaligned\right.
\end{equation*}
In this paper, we shall use the convention that $C$ denotes a generic constant, whose value may change from line to line. We shall write $C(\lambda_{1},\lambda_{2},\cdot\cdot\cdot,\lambda_{k})$ as the constant $C$ depends on the quantities $\lambda_{1},\lambda_{2},\cdot\cdot\cdot,\lambda_{k}$.
We also denote $\Psi\thickapprox \Upsilon$ if there exist two constants $C_{1}\leq C_{2}$ such that $C_{1}\Upsilon\leq \Psi\leq C_{2}\Upsilon$.

\vskip .1in
We state that the existence and uniqueness of local smooth solutions can be performed through the standard approach (see for example \cite{Constantinf,MB}). Thus, in order to complete the proof of Theorem \ref{Th1}, it is sufficient to establish {\it a priori} estimates that hold for any fixed $T>0$.

\vskip .1in
Keeping in mind the fact that when $\alpha=\gamma=0$ and $\beta>1$, the corresponding system admits a unique global regular solution  \cite{CWY,Agelas,Yedfsdf8,JZ01}. As a result, it suffices to consider the case $\beta\leq1$ (thus $\gamma>0$) in this section.
\vskip .1in

\begin{center}
$$\textbf{Case 1}:  \  \beta>1-\frac{\gamma}{2}\ \  \mbox{\rm with}\  \gamma\in(0,\,1] $$
\end{center}

\vskip .1in
In this case, we should keep in mind that $\beta>\frac{1}{2}$.
Now we present the basic $L^2$-energy estimate as follows.
\begin{lemma}\label{RMHDY1}
Assume $(v_{0},\,b_{0})$ satisfies the assumptions stated in
Theorem \ref{Th1}.
Then the corresponding solution $(v, b)$
of (\ref{eqrtyett01}) admits the following bound for any $t\in[0,\,T]$
\begin{align}\label{RMHDTQye01}
\|u(t)\|_{L^{2}}^{2}+\|\Lambda^{\gamma}u(t)\|_{L^{2}}^{2}
+\|b(t)\|_{L^{2}}^{2}+\int_{0}^{t}{\|\Lambda^{\beta}b(\tau)\|_{L^{2}}^{2}\,d\tau}\leq
C(v_{0},b_{0}).
\end{align}
\end{lemma}

\begin{proof}
Taking the inner product of $(\ref{eqrtyett01})_{1}$ with $u$ and
the inner product of $(\ref{eqrtyett01})_{2}$ with $b$, using the
divergence free property and summing them up, we have
\begin{align}\label{RMHDTQye02}
\frac{1}{2}\frac{d}{dt}(\|u(t)\|_{L^{2}}^{2}+\|\Lambda^{\gamma}u(t)\|_{L^{2}}^{2}
+\|b(t)\|_{L^{2}}^{2})+\|\Lambda^{\beta} b\|_{L^{2}}^{2}=0,
\end{align}
where we have used the following cancelations
\begin{align}\label{cancc01}
\int_{\mathbb{R}^{2}}{(b\cdot \nabla b)\cdot
u \,dx}+\int_{\mathbb{R}^{2}}{(b\cdot \nabla u) \cdot b\,dx}=0,\end{align}
\begin{align}\label{cancc02}
\int_{\mathbb{R}^{2}}{(u\cdot \nabla v)\cdot
u \,dx}+\int_{\mathbb{R}^{2}}{\Big(\sum_{j=1}^{2}v_{j}\nabla u_{j}\Big) \cdot u\,dx}=0.\end{align}
Integrating (\ref{RMHDTQye02}) in time yields the desired (\ref{RMHDTQye01}). This ends the proof of Lemma \ref{RMHDY1}.
\end{proof}

\vskip .1in
Next we will derive the following regularity estimate for $b$.

\begin{lemma}\label{RMHDY2}
Assume $(v_{0},\,b_{0})$ satisfies the assumptions stated in
Theorem \ref{Th1}. If $\beta>1-\frac{\gamma}{2}$ with $\gamma\in(0,\,1]$, then the corresponding solution $(v, b)$
of (\ref{eqrtyett01}) admits the following bound
\begin{align}\label{RMHDTQye03}
\sup_{t\in [0,\,T]}\|\Lambda^{\beta+\gamma-1}b(t)\|_{L^{2}}^{2}
+\int_{0}^{T}{\|\Lambda^{2\beta+\gamma-1}b(\tau)\|_{L^{2}}^{2}\,d\tau}\leq
C(T,v_{0},b_{0}).
\end{align}
In particular, due to $\beta>1-\frac{\gamma}{2}$, it holds from (\ref{RMHDTQye03}) that
\begin{align}\label{sdfTecfr1}
\int_{0}^{T}{\|b(\tau)\|_{L^{\infty}}^{2}\,d\tau}\leq
C(T,v_{0},b_{0}).
\end{align}
\end{lemma}

\begin{proof}
Applying $\Lambda^{\beta+\gamma-1}$ to the second equation of (\ref{eqrtyett01}) and multiplying it by $\Lambda^{\beta+\gamma-1}b$ yield
\begin{align}\label{RMHDTQye04}
\frac{1}{2}\frac{d}{dt}\|\Lambda^{\beta+\gamma-1}b(t)\|_{L^{2}}^{2}+
\|\Lambda^{2\beta+\gamma-1}b\|_{L^{2}}^{2}=J_{1}+J_{2},
\end{align}
where
$$J_{1}=\int_{\mathbb{R}^{2}}{\Lambda^{\beta+\gamma-1}(b\cdot \nabla u) \cdot \Lambda^{\beta+\gamma-1}b\,dx},\qquad J_{2}=-\int_{\mathbb{R}^{2}}{\Lambda^{\beta+\gamma-1}(u\cdot \nabla b) \cdot \Lambda^{\beta+\gamma-1}b\,dx}.$$
By $\nabla\cdot b=0$, Sobolev embedding inequalities and (\ref{yzzqq1}), we deduce
\begin{align}\label{cvfaddrt7y2}
J_{1}=&\int_{\mathbb{R}^{2}}{\Lambda^{\beta+\gamma-1}\nabla\cdot(b\otimes u) \cdot \Lambda^{\beta+\gamma-1}b\,dx}\nonumber\\
\leq&C\|\Lambda^{\gamma}(ub)\|_{L^{2}}\|\Lambda^{2\beta+\gamma-1}b\|_{L^{2}}
\nonumber\\
\leq&C(\|\Lambda^{\gamma}u\|_{L^{2}}\|b\|_{L^{\infty}}+
\|u\|_{L^{p_{0}}}\|\Lambda^{\gamma}b\|_{L^{\frac{2p_{0}}{p_{0}-2}}})
\|\Lambda^{2\beta+\gamma-1}b\|_{L^{2}}
\nonumber\\
\leq&C(\|u\|_{H^{\gamma}}\|b\|_{L^{2}}^{\frac{2\beta+\gamma-2}{2\beta+\gamma-1}}
\|\Lambda^{2\beta+\gamma-1}b\|_{L^{2}}^{\frac{1}{2\beta+\gamma-1}}+
\|u\|_{H^{\gamma}}
\|b\|_{L^{2}}^{\frac{(2\beta-1)p_{0}-2}{(2\beta+\gamma-1)p_{0}}}
\|\Lambda^{2\beta+\gamma-1}b\|_{L^{2}}^{\frac{\gamma p_{0}+2}{(2\beta+\gamma-1)p_{0}}})\nonumber\\&\times
\|\Lambda^{2\beta+\gamma-1}b\|_{L^{2}}
\nonumber\\
\leq&\frac{1}{4}\|\Lambda^{2\beta+\gamma-1}b\|_{L^{2}}^{2}
+C\|u\|_{H^{\gamma}}^{\frac{2(2\beta+\gamma-1)}{2\beta+\gamma-2}}\|b\|_{L^{2}}^{2}
+C\|u\|_{H^{\gamma}}^{\frac{2(2\beta+\gamma-1)p_{0}}{(2\beta-1)p_{0}-2}}
\|b\|_{L^{2}}^{2},
\end{align}
where $p_{0}>2$ satisfies
$$\frac{1-\gamma}{2}<\frac{1}{p_{0}}<\frac{2\beta-1}{2}.$$
We remark that as $\beta$ and $\gamma>0$ satisfy (\ref{condyet11}), the above $p_{0}$ actually would work.
Thanks to $\nabla\cdot u=0$, it also yields
\begin{align}J_{2}=&\int_{\mathbb{R}^{2}}{\Lambda^{\beta+\gamma-1}\nabla\cdot(u\otimes b) \cdot \Lambda^{\beta+\gamma-1}b\,dx}\nonumber\\
\leq&C\|\Lambda^{\gamma}(ub)\|_{L^{2}}\|\Lambda^{2\beta+\gamma-1}b\|_{L^{2}}
\nonumber\\
\leq&\frac{1}{4}\|\Lambda^{2\beta+\gamma-1}b\|_{L^{2}}^{2}
+C\|u\|_{H^{\gamma}}^{\frac{2(2\beta+\gamma-1)}{2\beta+\gamma-2}}\|b\|_{L^{2}}^{2}
+C\|u\|_{H^{\gamma}}^{\frac{2(2\beta+\gamma-1)p_{0}}{(2\beta-1)p_{0}-2}}
\|b\|_{L^{2}}^{2}.\nonumber
\end{align}
Inserting the above estimates into (\ref{RMHDTQye04}) implies
\begin{align}\label{RMHDTQye05}
\frac{d}{dt}\|\Lambda^{\beta+\gamma-1}b(t)\|_{L^{2}}^{2}+
\|\Lambda^{2\beta+\gamma-1}b\|_{L^{2}}^{2}\leq C(\|u\|_{H^{\gamma}}^{\frac{2(2\beta+\gamma-1)}{2\beta+\gamma-2}}
+\|u\|_{H^{\gamma}}^{\frac{2(2\beta+\gamma-1)p_{0}}{(2\beta-1)p_{0}-2}}
)\|b\|_{L^{2}}^{2}.
\end{align}
Recalling (\ref{RMHDTQye01}), one deduces from (\ref{RMHDTQye05}) that
$$\sup_{t\in [0,\,T]}\|\Lambda^{\beta+\gamma-1}b(t)\|_{L^{2}}^{2}
+\int_{0}^{T}{\|\Lambda^{2\beta+\gamma-1}b(\tau)\|_{L^{2}}^{2}\,d\tau}\leq
C(T,v_{0},b_{0}).$$
We thus end the proof of Lemma \ref{RMHDY2}.
\end{proof}

\vskip .1in
With the help of Lemma \ref{RMHDY2}, we continue to improve the regularity of $b$.

\begin{lemma}\label{RMHDY3}
Assume $(v_{0},\,b_{0})$ satisfies the assumptions stated in
Theorem \ref{Th1}. If $\beta>1-\frac{\gamma}{2}$ with $\gamma\in(0,\,1]$, then the corresponding solution $(v, b)$
of (\ref{eqrtyett01}) admits the following bound
\begin{align}\label{RMHDTQye06}
\sup_{t\in [0,\,T]}(\|v(t)\|_{H^{1}}^{2}+\|\Lambda^{2\gamma+\beta}b(t)\|_{L^{2}}^{2})
+\int_{0}^{T}{\|\Lambda^{2\gamma+2\beta}b(\tau)\|_{L^{2}}^{2}\,d\tau}\leq
C(T,v_{0},b_{0}).
\end{align}
In particular, we have
\begin{align}\label{sdfTecfr2}
\sup_{t\in [0,\,T]}\|b(t)\|_{L^{\infty}}
+\int_{0}^{T}{\|\nabla b(\tau)\|_{L^{\infty}}^{2}\,d\tau}\leq
C(T,v_{0},b_{0}).
\end{align}
\end{lemma}

\begin{proof}
Applying $\nabla^{\perp}:=(-\partial_{2},\,\partial_{1})^{\mathbb{T}}$ to the first equation of (\ref{eqrtyett01}), we show that the vorticity $\omega:=\nabla^{\perp}\cdot v\equiv \partial_{1} v_{2}-\partial_{2} v_{1}$ satisfies
\begin{align}\label{RMHDTQye07}
\partial_{t}\omega+(u\cdot\nabla)\omega=\nabla^{\perp}\nabla\cdot(b\otimes b).
\end{align}
It follows from \eqref{RMHDTQye07} that
\begin{align}\label{RMHDTQye08}
\frac{1}{2}\frac{d}{dt}\|\omega(t)\|_{L^{2}}^{2}&\leq C\|\nabla^{\perp}\nabla\cdot(b\otimes b)\|_{L^{2}}\|\omega\|_{L^{2}}\nonumber\\&\leq C\| b\|_{L^{\infty}}\|\Delta b\|_{L^{2}}\|\omega\|_{L^{2}}
\nonumber\\&\leq C\| b\|_{L^{\infty}}\|b\|_{L^{2}}^{\frac{\gamma+\beta-1}{\gamma+\beta}}
\|\Lambda^{2\gamma+2\beta}b\|_{L^{2}}
^{\frac{1}{\gamma+\beta}}\|\omega\|_{L^{2}}
\nonumber\\&\leq \frac{1}{4}\|\Lambda^{2\gamma+2\beta}b\|_{L^{2}}^{2}+C(1+\| b\|_{L^{\infty}}^{2})(1+\|\omega\|_{L^{2}}^{2}).
\end{align}
Applying $\Lambda^{2\gamma+\beta}$ to the second equation of (\ref{eqrtyett01}) and multiplying it by $\Lambda^{2\gamma+\beta}b$, one has
\begin{align}\label{RMHDTQye11}
\frac{1}{2}\frac{d}{dt}\|\Lambda^{2\gamma+\beta}b(t)\|_{L^{2}}^{2}+
\|\Lambda^{2\gamma+2\beta}b\|_{L^{2}}^{2}=J_{3}+J_{4},
\end{align}
where
$$J_{3}=\int_{\mathbb{R}^{2}}{\Lambda^{2\gamma+\beta}(b\cdot \nabla u) \cdot \Lambda^{2\gamma+\beta}b\,dx},\qquad J_{4}=-\int_{\mathbb{R}^{2}}{\Lambda^{2\gamma+\beta}(u\cdot \nabla b) \cdot \Lambda^{2\gamma+\beta}b\,dx}.$$
By the argument adopted in dealing with (\ref{cvfaddrt7y2}), it leads to
\begin{align}\label{RMHDTQye12}
J_{3}=&\int_{\mathbb{R}^{2}}{\Lambda^{2\gamma+\beta}\nabla\cdot(b\otimes u) \cdot \Lambda^{2\gamma+\beta}b\,dx}\nonumber\\
\leq&C\|\Lambda^{2\gamma+1}(ub)\|_{L^{2}}\|\Lambda^{2\gamma+2\beta} b\|_{L^{2}}
\nonumber\\
\leq&C(\|\Lambda^{2\gamma+1}u\|_{L^{2}}\|b\|_{L^{\infty}}+
\|u\|_{L^{p_{0}}}\|\Lambda^{2\gamma+1}b\|_{L^{\frac{2p_{0}}{p_{0}-2}}})
\|\Lambda^{2\gamma+2\beta} b\|_{L^{2}}
\nonumber\\
\leq&C(\|\nabla v\|_{L^{2}}\|b\|_{L^{\infty}}+
\|u\|_{H^{\gamma}}
\|b\|_{L^{2}}^{\frac{(2\beta-1)p_{0}-2}{2(\gamma+\beta)p_{0}}}
\|\Lambda^{2\gamma+2\beta} b\|_{L^{2}}^{\frac{(2\gamma+1)p_{0}+2}{2(\gamma+\beta)p_{0}}})
\|\Lambda^{2\gamma+2\beta} b\|_{L^{2}}
\nonumber\\
\leq&\frac{1}{8}\|\Lambda^{2\gamma+2\beta} b\|_{L^{2}}^{2}
+C\|b\|_{L^{\infty}}^{2}\|\omega\|_{L^{2}}^{2}
+C\|u\|_{H^{\gamma}}^{\frac{4(\beta+\gamma)p_{0}}{(2\beta-1)p_{0}-2}}
\|b\|_{L^{2}}^{2},
\end{align}
\begin{align}\label{RMHDTQye13}
J_{4}=&-\int_{\mathbb{R}^{2}}{\Lambda^{2\gamma+\beta}\nabla\cdot(u\otimes b) \cdot \Lambda^{2\gamma+\beta}b\,dx}\nonumber\\
\leq&C\|\Lambda^{2\gamma+1}(ub)\|_{L^{2}}\|\Lambda^{2\gamma+2\beta} b\|_{L^{2}}
\nonumber\\
\leq&\frac{1}{8}\|\Lambda^{2\gamma+2\beta} b\|_{L^{2}}^{2}
+C\|b\|_{L^{\infty}}^{2}\|\omega\|_{L^{2}}^{2}
+C\|u\|_{H^{\gamma}}^{\frac{4(\beta+\gamma)p_{0}}{(2\beta-1)p_{0}-2}}
\|b\|_{L^{2}}^{2}.
\end{align}
Putting (\ref{RMHDTQye12}) and (\ref{RMHDTQye13}) into (\ref{RMHDTQye11}) gives
\begin{align}\label{cdftade117}
\frac{1}{2}\frac{d}{dt}\|\Lambda^{2\gamma+\beta}b(t)\|_{L^{2}}^{2}+
\frac{3}{4}\|\Lambda^{2\gamma+2\beta}b\|_{L^{2}}^{2}\leq C\|b\|_{L^{\infty}}^{2}\|\omega\|_{L^{2}}^{2}
+C\|u\|_{H^{\gamma}}^{\frac{4(\beta+\gamma)p_{0}}{(2\beta-1)p_{0}-2}}
\|b\|_{L^{2}}^{2},
\end{align}
which along with \eqref{RMHDTQye08} further leads to
\begin{align}\label{RMHDTQye14}
\frac{d}{dt}(\|\Lambda^{2\gamma+\beta}b(t)\|_{L^{2}}^{2}+\|\omega(t)\|_{L^{2}}^{2})+
\|\Lambda^{2\gamma+2\beta}b\|_{L^{2}}^{2}\leq &C(1+\| b\|_{L^{\infty}}^{2})(1+\|\omega\|_{L^{2}}^{2})
\nonumber\\&+C\|u\|_{H^{\gamma}}^{\frac{4(\beta+\gamma)p_{0}}{(2\beta-1)p_{0}-2}}
\|b\|_{L^{2}}^{2}.
\end{align}
Thanks to \eqref{sdfTecfr1}, we deduce by applying the Gronwall inequality to \eqref{RMHDTQye14} that
\begin{align} \label{RMHDTQye15}
\sup_{t\in [0,\,T]}(\|\nabla v(t)\|_{L^{2}}^{2}+\|\Lambda^{2\gamma+\beta}b(t)\|_{L^{2}}^{2})
+\int_{0}^{T}{\|\Lambda^{2\gamma+2\beta}b(\tau)\|_{L^{2}}^{2}\,d\tau}\leq
C(T,v_{0},b_{0}),
\end{align}
where we have used the simple fact $\|\omega\|_{L^{2}}\thickapprox \|\nabla v\|_{L^{2}}$. In order to obtain the estimate of $\|v(t)\|_{L^{2}}$, we take the $L^2$-inner product of $\eqref{eqrtyett01}_{1}$ with $v$ to get
\begin{align*}
\frac{1}{2}\frac{d}{dt}\|v(t)\|_{L^{2}}^{2}&=-\int_{\mathbb{R}^{2}}{\Big(\sum_{j=1}^{2}v_{j}\nabla u_{j}\Big) \cdot v\,dx}+\int_{\mathbb{R}^{2}}{(b\cdot \nabla b) \cdot v\,dx}\nonumber\\&\leq C\|\nabla u\|_{L^{2}}\|v\|_{L^{4}}^{2}+C\| b\|_{L^{\infty}}\|\nabla b\|_{L^{2}}\|v\|_{L^{2}}\nonumber\\&\leq C\|\nabla v\|_{L^{2}}(\|v\|_{L^{2}}\|\nabla v\|_{L^{2}})+C\|b\|_{H^{2\gamma+\beta}}^{2}\|v\|_{L^{2}}
\nonumber\\&\leq C(\|\nabla v\|_{L^{2}}^{2}+\|b\|_{H^{2\gamma+\beta}}^{2})\|v\|_{L^{2}},
\end{align*}
which yields
\begin{align}\label{ccfert45cd3}
 \frac{d}{dt}\|v(t)\|_{L^{2}} \leq C(\|\nabla v\|_{L^{2}}^{2}+\|b\|_{H^{2\gamma+\beta}}^{2}).
\end{align}
Thanks to \eqref{RMHDTQye15}, we obtain
\begin{align} \label{RMHDTQye16}
\sup_{t\in [0,\,T]}\|v(t)\|_{L^{2}}\leq C(T,v_{0},b_{0}).
\end{align}
The desired estimate \eqref{RMHDTQye06} follows by combining \eqref{RMHDTQye15} and \eqref{RMHDTQye16}.
This completes the proof of Lemma \ref{RMHDY3}.
\end{proof}

\vskip .1in
Now we are in the position to show the following crucial bound which allows us to derive the higher regularity of $(v,\,b)$.
\begin{lemma}\label{RMHDY4}
Assume $(v_{0},\,b_{0})$ satisfies the assumptions stated in
Theorem \ref{Th1}. If $\beta>1-\frac{\gamma}{2}$ with $\gamma\in(0,\,1]$, then the corresponding solution $(v, b)$
of (\ref{eqrtyett01}) admits the following bound
\begin{align*}
\sup_{t\in [0,\,T]}\|\omega(t)\|_{L^{\infty}}\leq C(T,v_{0},b_{0}).
\end{align*}
\end{lemma}

\begin{proof}
We rewrite the equation  $\eqref{eqrtyett01}_{2}$ as follows
$$\partial_{t}b+\Lambda^{2\beta} b=\nabla\cdot(b\otimes u-u\otimes b).$$
Applying $\Lambda^{2\gamma+1}$ to both sides of the above equation, we have
\begin{align}\label{RMHDTQye19}
\partial_{t}\Lambda^{2\gamma+1}b+\Lambda^{2\beta}\Lambda^{2\gamma+1}b
=\Lambda^{2\gamma+1}\nabla\cdot(b\otimes u-u\otimes b).
\end{align}
Applying Lemma \ref{afasqw6sf} to \eqref{RMHDTQye19}, it yields that for any $\epsilon\in (0,\,2\beta-1)$ and for any $q\in\big(\frac{1}{2\beta-1-\epsilon},\,\infty\big)$
\begin{align}
\|\Lambda^{2\beta+2\gamma-\epsilon}b\|_{L_{T}^{1}L^{q}}
&=\|\Lambda^{2\beta-1-\epsilon}\Lambda^{2\gamma+1}b\|_{L_{T}^{1}L^{q}} \nonumber\\
&\leq C(T,v_{0},b_{0})+C\|\Lambda^{2\gamma+1}(ub)\|_{L_{T}^{1}L^{q}}
 \nonumber\\
&\leq C(T,v_{0},b_{0})+C\|u\|_{L_{T}^{\infty}L^{2q}} \|\Lambda^{2\gamma+1}b\|_{L_{T}^{1}L^{2q}}+C\|b\|_{L_{T}^{\infty}L^{\infty}}
\|\Lambda^{2\gamma+1}u\|_{L_{T}^{1}L^{q}}
\nonumber\\
&\leq C(T,v_{0},b_{0})+C\|v\|_{L_{T}^{\infty}H^{1}}
\|b\|_{L_{T}^{1}L^{q}}^{\frac{(2\beta-1-\epsilon)q-1}{
(2\beta+2\gamma-\epsilon)q}}\|\Lambda^{2\beta+2\gamma-\epsilon}b\|_{L_{T}^{1}L^{q}}^{\frac{(2\gamma+1)q+1}{
(2\beta+2\gamma-\epsilon)q}}\nonumber\\&+C\|b\|_{L_{T}^{\infty}L^{\infty}}
\|\omega\|_{L_{T}^{1}L^{q}}
\nonumber\\
&\leq \frac{1}{2}\|\Lambda^{2\beta+2\gamma-\epsilon}b\|_{L_{T}^{1}L^{q}}+
C(T,v_{0},b_{0})+C(T,v_{0},b_{0})
\|\omega\|_{L_{T}^{1}L^{q}},\nonumber
\end{align}
which leads to
\begin{align}\label{RMHDTQye20}
\|\Lambda^{2\beta+2\gamma-\epsilon}b\|_{L_{T}^{1}L^{q}}\leq
C(T,v_{0},b_{0})+C(T,v_{0},b_{0})
\|\omega\|_{L_{T}^{1}L^{q}}.
\end{align}
In order to close the above inequality, we multiply \eqref{RMHDTQye07} by $|\omega|^{q-2}\omega$ and integrate it over whole space to deduce
\begin{align}\label{RMHDTQye21}
\frac{1}{q}\frac{d}{dt}\|\omega(t)\|_{L^{q}}^{q}&\leq C\|\nabla^{\perp}\nabla\cdot(b\otimes b)\|_{L^{q}}\|\omega\|_{L^{q}}^{q-1}\nonumber\\&\leq C\| b\|_{L^{\infty}}\|\Delta b\|_{L^{q}}\|\omega\|_{L^{q}}^{q-1}
\nonumber\\&\leq C\| b\|_{L^{\infty}}\|b\|_{L^{q}}^{1-\frac{2}{2\gamma+2\beta-\epsilon}}
\|\Lambda^{2\beta+2\gamma-\epsilon}b\|_{L^{q}}
^{\frac{2}{2\beta+2\gamma-\epsilon}}\|\omega\|_{L^{q}}^{q-1}
\nonumber\\&\leq C(T,v_{0},b_{0})
\|\Lambda^{2\beta+2\gamma-\epsilon}b\|_{L^{q}}
^{\frac{2}{2\beta+2\gamma-\epsilon}}\|\omega\|_{L^{q}}^{q-1},
\end{align}
where we further restrict $\epsilon$ to satisfy $\epsilon\in (0,\,2\beta+2\gamma-2)$. Therefore, it is easy to deduce from \eqref{RMHDTQye21} that
\begin{align}\label{RMHDTQye22}
\frac{d}{dt}\|\omega(t)\|_{L^{q}}\leq C(T,v_{0},b_{0})
\|\Lambda^{2\beta+2\gamma-\epsilon}b\|_{L^{q}}
^{\frac{2}{2\beta+2\gamma-\epsilon}}.
\end{align}
Integrating \eqref{RMHDTQye22} in time yields
\begin{align}\label{RMHDTQye23}
\|\omega(t)\|_{L_{T}^{\infty}L^{q}}\leq C(T,v_{0},b_{0})+C(T,v_{0},b_{0})
\|\Lambda^{2\beta+2\gamma-\epsilon}b\|_{L_{T}^{1}L^{q}}
^{\frac{2}{2\beta+2\gamma-\epsilon}}.
\end{align}
Combining \eqref{RMHDTQye20} and \eqref{RMHDTQye23}, we derive
\begin{align}
\|\Lambda^{2\beta+2\gamma-\epsilon}b\|_{L_{T}^{1}L^{q}}&\leq
C(T,v_{0},b_{0})+C(T,v_{0},b_{0})
\|\omega\|_{L_{T}^{1}L^{q}}\nonumber\\&\leq
C(T,v_{0},b_{0})+C(T,v_{0},b_{0})
\|\omega\|_{L_{T}^{\infty}L^{q}}\nonumber\\&\leq
C(T,v_{0},b_{0})+C(T,v_{0},b_{0})
\|\Lambda^{2\beta+2\gamma-\epsilon}b\|_{L_{T}^{1}L^{q}}
^{\frac{2}{2\beta+2\gamma-\epsilon}}\nonumber\\&\leq
\frac{1}{2}\|\Lambda^{2\beta+2\gamma-\epsilon}b\|_{L_{T}^{1}L^{q}}+C(T,v_{0},b_{0})
,\nonumber
\end{align}
which allows us to conclude
\begin{align}
\|\Lambda^{2\beta+2\gamma-\epsilon}b\|_{L_{T}^{1}L^{q}}\leq
C(T,v_{0},b_{0}).\nonumber
\end{align}
Keeping in mind that $\beta>1-\frac{\gamma}{2}$, we may take $q$ suitably large such that
\begin{align}\label{RMHDTQye24}
\|\nabla^{2}b\|_{L_{T}^{1}L^{\infty}}\leq C\|b\|_{L_{T}^{1}L^{\infty}}+ C\|\Lambda^{2\beta+2\gamma-\epsilon}b\|_{L_{T}^{1}L^{q}}\leq
C(T,v_{0},b_{0}).
\end{align}
Combing back to \eqref{RMHDTQye21}, we can deduce
\begin{align*}
\frac{1}{q}\frac{d}{dt}\|\omega(t)\|_{L^{q}}^{q}&\leq C\|\nabla^{\perp}\nabla\cdot(b\otimes b)\|_{L^{q}}\|\omega\|_{L^{q}}^{q-1}\nonumber\\&\leq C(\|b\|_{L^{\infty}}\|\nabla^{2} b\|_{L^{q}}+\|\nabla b\|_{L^{\infty}}\|\nabla b\|_{L^{q}})\|\omega\|_{L^{q}}^{q-1},
\end{align*}
which gives
\begin{align*}
\frac{d}{dt}\|\omega(t)\|_{L^{q}}\leq C(\|b\|_{L^{\infty}}\|\nabla^{2} b\|_{L^{q}}+\|\nabla b\|_{L^{\infty}}\|\nabla b\|_{L^{q}}).
\end{align*}
Letting $p\rightarrow\infty$, one has
\begin{align}\label{RMHDTQye25}
\frac{d}{dt}\|\omega(t)\|_{L^{\infty}}\leq C(\|b\|_{L^{\infty}}\|\nabla^{2} b\|_{L^{\infty}}+\|\nabla b\|_{L^{\infty}}^{2}).
\end{align}
Making use of \eqref{sdfTecfr2} and \eqref{RMHDTQye24}, we get by integrating \eqref{RMHDTQye25} in time that
\begin{align*}
\sup_{t\in [0,\,T]}\|\omega(t)\|_{L^{\infty}}\leq C(T,v_{0},b_{0}).
\end{align*}
We therefore complete the proof of Lemma \ref{RMHDY4}.
\end{proof}

\vskip .1in

\begin{center}
$$\textbf{Case 2}:  \  \beta>1-\frac{\gamma}{2}\ \  \mbox{\rm with}\  \gamma\in(1,\,2] $$
\end{center}

\vskip .1in

In this case, we first state that the case $\alpha=\beta=0,\, \gamma=2$ will be considered in Theorem \ref{addTh2}. Therefore, it suffices to consider the case
$\beta>1-\frac{\gamma}{2}$ with $\gamma\in(1,\,2)$. Without loss of generality, we may assume $\beta\leq\frac{1}{2}$ as $\beta>\frac{1}{2}$ can be handled as that of \textbf{Case 1}.
We remark that in this case, we still have the estimate \eqref{RMHDTQye01}. Now we will establish the following result as Lemma \ref{RMHDY2}.
\begin{lemma}\label{ca2lena25}
Assume $(v_{0},\,b_{0})$ satisfies the assumptions stated in
Theorem \ref{Th1}. If $\beta>1-\frac{\gamma}{2}$ with $\gamma\in(1,\,2)$, then the corresponding solution $(v, b)$
of (\ref{eqrtyett01}) admits the following bound
\begin{align}\label{ca2yeryet28}
\sup_{t\in [0,\,T]}\|\Lambda^{\beta+\gamma-1}b(t)\|_{L^{2}}^{2}
+\int_{0}^{T}{\|\Lambda^{2\beta+\gamma-1}b(\tau)\|_{L^{2}}^{2}\,d\tau}\leq
C(T,v_{0},b_{0}).
\end{align}
In particular, due to $\beta>1-\frac{\gamma}{2}$, it holds from (\ref{ca2yeryet28}) that
\begin{align}\label{ca2yeryet29}
\int_{0}^{T}{(\|\nabla b(\tau)\|_{L^{2}}^{2}+\|b(\tau)\|_{L^{\infty}}^{2})\,d\tau}\leq
C(T,v_{0},b_{0}).
\end{align}
\end{lemma}

\begin{proof}
Recalling \eqref{RMHDTQye04}, we get
\begin{align*}
\frac{1}{2}\frac{d}{dt}\|\Lambda^{\beta+\gamma-1}b(t)\|_{L^{2}}^{2}+
\|\Lambda^{2\beta+\gamma-1}b\|_{L^{2}}^{2}=J_{1}+J_{2},
\end{align*}
where
$$J_{1}=\int_{\mathbb{R}^{2}}{\Lambda^{\beta+\gamma-1}(b\cdot \nabla u) \cdot \Lambda^{\beta+\gamma-1}b\,dx},\qquad J_{2}=-\int_{\mathbb{R}^{2}}{\Lambda^{\beta+\gamma-1}(u\cdot \nabla b) \cdot \Lambda^{\beta+\gamma-1}b\,dx}.$$
It follows from (\ref{yzzqq1}) that
\begin{align*}
J_{1}
\leq&C\|\Lambda^{\gamma-1}(b\cdot \nabla u)\|_{L^{2}}\|\Lambda^{2\beta+\gamma-1}b\|_{L^{2}}
\nonumber\\
\leq&C(\|\Lambda^{\gamma}u\|_{L^{2}}\|b\|_{L^{\infty}}+
\|\nabla u\|_{L^{\frac{2}{2-\gamma}}}\|\Lambda^{\gamma-1}b\|_{L^{\frac{2}{\gamma-1}}})
\|\Lambda^{2\beta+\gamma-1}b\|_{L^{2}}
\nonumber\\
\leq&C\|\Lambda^{\gamma}u\|_{L^{2}}\|b\|_{L^{2}}^{\frac{2\beta+\gamma-2}{2\beta+\gamma-1}}
\|\Lambda^{2\beta+\gamma-1}b\|_{L^{2}}^{\frac{1}{2\beta+\gamma-1}}
\|\Lambda^{2\beta+\gamma-1}b\|_{L^{2}}
\nonumber\\
\leq&\frac{1}{4}\|\Lambda^{2\beta+\gamma-1}b\|_{L^{2}}^{2}
+C\|\Lambda^{\gamma}u\|_{L^{2}}^{\frac{2(2\beta+\gamma-1)}{2\beta+\gamma-2}}
\|b\|_{L^{2}}^{2}.
\end{align*}
According to \eqref{yzz}, it leads to
\begin{align*}
J_{4}=&-\int_{\mathbb{R}^{2}}{[\Lambda^{\beta+\gamma-1},\, u\cdot\nabla] b\cdot \Lambda^{\beta+\gamma-1}b\,dx}\nonumber\\
\leq&C\|[\Lambda^{\beta+\gamma-1},\, u\cdot\nabla] b\|_{L^{\frac{2}{1+\beta}}}\|\Lambda^{\beta+\gamma-1} b\|_{L^{\frac{2}{1-\beta}}}
\nonumber\\
\leq&C(\|\nabla u\|_{L^{p}}\|\Lambda^{\beta+\gamma-1} b\|_{L^{\frac{2p}{(1+\beta)p-2}}}+\|\nabla b\|_{L^{2}}\|\Lambda^{\beta+\gamma-1}u\|_{L^{\frac{2}{\beta}}})
\|\Lambda^{2\beta+\gamma-1} b\|_{L^{2}}
\nonumber\\
\leq&C(\|u\|_{H^{\gamma}}\|b\|_{L^{2}}^{1-\lambda}\|\Lambda^{2\beta+\gamma-1} b\|_{L^{2}}^{\lambda}+\|b\|_{L^{2}}^{\frac{2\beta+\gamma-2}{2\beta+\gamma-1}}
\|\Lambda^{2\beta+\gamma-1} b\|_{L^{2}}^{\frac{1}{2\beta+\gamma-1}}\|\Lambda^{\gamma}u\|_{L^{2}})\|\Lambda^{2\beta+\gamma-1} b\|_{L^{2}}
\nonumber\\
\leq&\frac{1}{8}\|\Lambda^{2\beta+\gamma-1} b\|_{L^{2}}^{2}
+C(\|u\|_{H^{\gamma}}^{\frac{2}{1-\lambda}}+\|\Lambda^{\gamma}u
\|_{L^{2}}^{\frac{2(2\beta+\gamma-1)}{2\beta+\gamma-2}})
\|b\|_{L^{2}}^{2},
\end{align*}
where $p$ and $\lambda$ satisfy
$$\max\left\{\frac{2-\gamma}{2},\ \frac{\beta}{2}\right\}<\frac{1}{p}<\min\left\{\frac{1}{2},\ \beta\right\},\qquad \lambda=\frac{(\gamma-1) p+2}{(2\beta+\gamma-1)p}\in (0,\,1).$$
We obtain by combining all the above estimates
\begin{align*}
\frac{d}{dt}\|\Lambda^{\beta+\gamma-1}b(t)\|_{L^{2}}^{2}+
\|\Lambda^{2\beta+\gamma-1}b\|_{L^{2}}^{2}\leq C(\|u\|_{H^{\gamma}}^{\frac{2}{1-\lambda}}+\|\Lambda^{\gamma}u
\|_{L^{2}}^{\frac{2(2\beta+\gamma-1)}{2\beta+\gamma-2}})
\|b\|_{L^{2}}^{2}.
\end{align*}
Using (\ref{RMHDTQye01}), we deduce
$$\sup_{t\in [0,\,T]}\|\Lambda^{\beta+\gamma-1}b(t)\|_{L^{2}}^{2}
+\int_{0}^{T}{\|\Lambda^{2\beta+\gamma-1}b(\tau)\|_{L^{2}}^{2}\,d\tau}\leq
C(T,v_{0},b_{0}),$$
which is \eqref{ca2yeryet28}. This completes the proof of Lemma \ref{ca2lena25}.
\end{proof}

\vskip .1in
Next, we will derive the following result as Lemma \ref{RMHDY3}.

\begin{lemma}\label{ca2lena26}
Assume $(v_{0},\,b_{0})$ satisfies the assumptions stated in
Theorem \ref{Th1}. If $\beta>1-\frac{\gamma}{2}$ with $\gamma\in(1,\,2)$, then the corresponding solution $(v, b)$
of (\ref{eqrtyett01}) admits the following bound
\begin{align}\label{ca2yeryet30}
\sup_{t\in [0,\,T]}(\|v(t)\|_{H^{1}}^{2}+\|\Lambda^{2\gamma+\beta}b(t)\|_{L^{2}}^{2})
+\int_{0}^{T}{\|\Lambda^{2\gamma+2\beta}b(\tau)\|_{L^{2}}^{2}\,d\tau}\leq
C(T,v_{0},b_{0}).
\end{align}
In particular, there holds
\begin{align}\label{ca2yeryet31}
\sup_{t\in [0,\,T]}\|b(t)\|_{L^{\infty}}
+\int_{0}^{T}{\|\nabla b(\tau)\|_{L^{\infty}}^{2}\,d\tau}\leq
C(T,v_{0},b_{0}).
\end{align}
\end{lemma}

\begin{proof}
Recalling \eqref{RMHDTQye08}, we have
\begin{align}\label{ca2yeryet32}
\frac{1}{2}\frac{d}{dt}\|\omega(t)\|_{L^{2}}^{2} \leq \frac{1}{4}\|\Lambda^{2\gamma+2\beta}b\|_{L^{2}}^{2}+C(1+\| b\|_{L^{\infty}}^{2})(1+\|\omega\|_{L^{2}}^{2}).
\end{align}
We have from \eqref{RMHDTQye11} that
\begin{align}\label{ca2yeryet33}
\frac{1}{2}\frac{d}{dt}\|\Lambda^{2\gamma+\beta}b(t)\|_{L^{2}}^{2}+
\|\Lambda^{2\gamma+2\beta}b\|_{L^{2}}^{2}=J_{3}+J_{4},
\end{align}
where
$$J_{3}=\int_{\mathbb{R}^{2}}{\Lambda^{2\gamma+\beta}(b\cdot \nabla u) \cdot \Lambda^{2\gamma+\beta}b\,dx},\qquad J_{4}=-\int_{\mathbb{R}^{2}}{\Lambda^{2\gamma+\beta}(u\cdot \nabla b) \cdot \Lambda^{2\gamma+\beta}b\,dx}.$$
Making use of \eqref{yzzqq1}, one deduces
\begin{align*}
J_{3}
\leq&C\|\Lambda^{2\gamma}(b\nabla u)\|_{L^{2}}\|\Lambda^{2\gamma+2\beta} b\|_{L^{2}}
\nonumber\\
\leq&C(\|b\|_{L^{\infty}}\|\Lambda^{2\gamma+1}u\|_{L^{2}}+
\|\nabla u\|_{L^{\frac{2}{2-\gamma}}}\|\Lambda^{2\gamma}b\|_{L^{\frac{2}{\gamma-1}}})
\|\Lambda^{2\gamma+2\beta} b\|_{L^{2}}
\nonumber\\
\leq&C(\|\nabla v\|_{L^{2}}\|b\|_{L^{\infty}}+
\|\Lambda^{\gamma}u\|_{L^{2}}
\|b\|_{L^{2}}^{\frac{2\beta+\gamma-2}{2(\beta+\gamma)}}
\|\Lambda^{2\gamma+2\beta} b\|_{L^{2}}^{\frac{2+\gamma}{2(\beta+\gamma)}})
\|\Lambda^{2\gamma+2\beta} b\|_{L^{2}}
\nonumber\\
\leq&\frac{1}{8}\|\Lambda^{2\gamma+2\beta} b\|_{L^{2}}^{2}
+C\|b\|_{L^{\infty}}^{2}\|\omega\|_{L^{2}}^{2}
+C\|\Lambda^{\gamma}u\|_{L^{2}}^{\frac{4(\beta+\gamma)}{2\beta+\gamma-2}}
\|b\|_{L^{2}}^{2}.
\end{align*}
For $J_{4}$, we have by \eqref{yzz} that
\begin{align*}
J_{4}=&-\int_{\mathbb{R}^{2}}{[\Lambda^{2\gamma+\beta},\, u\cdot\nabla] b\cdot \Lambda^{2\gamma+\beta}b\,dx}\nonumber\\
\leq&C\|[\Lambda^{2\gamma+\beta},\, u\cdot\nabla] b\|_{L^{\frac{2}{1+\beta}}}\|\Lambda^{2\gamma+\beta} b\|_{L^{\frac{2}{1-\beta}}}
\nonumber\\
\leq&C(\|\nabla u\|_{L^{p}}\|\Lambda^{2\gamma+\beta} b\|_{L^{\frac{2p}{(1+\beta)p-2}}}+\|\nabla b\|_{L^{2}}\|\Lambda^{2\gamma+\beta}u\|_{L^{\frac{2}{\beta}}})\|\Lambda^{2\gamma+2\beta} b\|_{L^{2}}
\nonumber\\
\leq&C(\|u\|_{H^{\gamma}}\|b\|_{L^{2}}^{1-\theta}\|\Lambda^{2\gamma+2\beta} b\|_{L^{2}}^{\theta}+\|\nabla b\|_{L^{2}}\|\Lambda^{2\gamma+1}u\|_{L^{2}})\|\Lambda^{2\gamma+2\beta} b\|_{L^{2}}
\nonumber\\
\leq&\frac{1}{8}\|\Lambda^{2\gamma+2\beta} b\|_{L^{2}}^{2}
+C\|\nabla b\|_{L^{2}}^{2}\|\omega\|_{L^{2}}^{2}
+C\|u\|_{H^{\gamma}}^{\frac{2}{1-\theta}}
\|b\|_{L^{2}}^{2},
\end{align*}
where $p$ and $\theta$ satisfy
$$\max\left\{\frac{2-\gamma}{2},\ \frac{\beta}{2}\right\}<\frac{1}{p}<\min\left\{\frac{1}{2},\ \beta\right\},\qquad \theta=\frac{\gamma p+1}{(\beta+\gamma)p}\in (0,\,1).$$
Substituting the above two estimates back into (\ref{ca2yeryet33}) gives
\begin{align*}
\frac{1}{2}\frac{d}{dt}\|\Lambda^{2\gamma+\beta}b(t)\|_{L^{2}}^{2}+
\frac{3}{4}\|\Lambda^{2\gamma+2\beta}b\|_{L^{2}}^{2}&\leq C(\|b\|_{L^{\infty}}^{2}+\|\nabla b\|_{L^{2}}^{2})\|\omega\|_{L^{2}}^{2}
\nonumber\\&+C(\|\Lambda^{\gamma}u\|_{L^{2}}^{\frac{4(\beta+\gamma)}{2\beta+\gamma-2}}
+\|u\|_{H^{\gamma}}^{\frac{2}{1-\theta}})
\|b\|_{L^{2}}^{2},
\end{align*}
which along with \eqref{ca2yeryet32} gives
\begin{align*}
\frac{d}{dt}(\|\Lambda^{2\gamma+\beta}b(t)\|_{L^{2}}^{2}+\|\omega(t)\|_{L^{2}}^{2})+
\|\Lambda^{2\gamma+2\beta}b\|_{L^{2}}^{2}\leq &C(1+\| b\|_{L^{\infty}}^{2}+\|\nabla b\|_{L^{2}}^{2})(1+\|\omega\|_{L^{2}}^{2})
\nonumber\\&+C(\|\Lambda^{\gamma}u\|_{L^{2}}^{\frac{4(\beta+\gamma)}{2\beta+\gamma-2}}
+\|u\|_{H^{\gamma}}^{\frac{2}{1-\theta}})
\|b\|_{L^{2}}^{2}.
\end{align*}
Thanks to \eqref{RMHDY1}, \eqref{ca2yeryet29} and the Gronwall inequality, we have that
\begin{align} \label{ca2yeryet35}
\sup_{t\in [0,\,T]}(\|\nabla v(t)\|_{L^{2}}^{2}+\|\Lambda^{2\gamma+\beta}b(t)\|_{L^{2}}^{2})
+\int_{0}^{T}{\|\Lambda^{2\gamma+2\beta}b(\tau)\|_{L^{2}}^{2}\,d\tau}\leq
C(T,v_{0},b_{0}).
\end{align}
Coming back to \eqref{ccfert45cd3}, we get
\begin{align*}
 \frac{d}{dt}\|v(t)\|_{L^{2}} \leq C(\|\nabla v\|_{L^{2}}^{2}+\|b\|_{H^{2\gamma+\beta}}^{2}).
\end{align*}
Using \eqref{ca2yeryet35}, it leads to
\begin{align*}
\sup_{t\in [0,\,T]}\|v(t)\|_{L^{2}}\leq C(T,v_{0},b_{0}).
\end{align*}
This completes the proof of Lemma \ref{ca2lena26}.
\end{proof}

\vskip .1in
Now let us to show the following crucial estimate.
\begin{lemma}\label{ca2lena27}
Assume $(v_{0},\,b_{0})$ satisfies the assumptions stated in
Theorem \ref{Th1}. If $\beta>1-\frac{\gamma}{2}$ with $\gamma\in(1,\,2)$, then the corresponding solution $(v, b)$
of (\ref{eqrtyett01}) admits the following bound
\begin{align*}
\sup_{t\in [0,\,T]} \|\omega(t)\|_{L^{\infty}} \leq
C(T,v_{0},b_{0}).
\end{align*}
\end{lemma}

\begin{proof}
The proof is different from Lemma \ref{RMHDY4} due to the fact that we only consider the case $\beta\leq\frac{1}{2}$. Fortunately, the bound \eqref{ca2yeryet30} allows us to conclude the desired estimate. It follows from \eqref{RMHDTQye07} that
\begin{align*}
\frac{d}{dt}\|\omega(t)\|_{L^{\infty}}\leq \|\nabla^{2}(bb)\|_{L^{\infty}}.
\end{align*}
Direct computations yield due to $\beta>1-\frac{\gamma}{2}$ with $\gamma\in(1,\,2)$
\begin{align*}
\|\nabla^{2}(bb)\|_{L^{\infty}}&\leq  C(\|bb\|_{L^{2}}+\|\Lambda^{2\gamma+2\beta}(bb)\|_{L^{2}})\nonumber\\
&\leq  C(\|b\|_{L^{\infty}}\|b\|_{L^{2}}+\|b\|_{L^{\infty}}
\|\Lambda^{2\gamma+2\beta}b\|_{L^{2}}).
\end{align*}
Thanks to \eqref{ca2yeryet30} and \eqref{ca2yeryet31}, it implies
$$\int_{0}^{T}{\|\nabla^{2}(bb)(\tau)\|_{L^{\infty}}\,d\tau}\leq
C(T,v_{0},b_{0}).$$
Consequently, we have
$$\sup_{t\in [0,\,T]} \|\omega(t)\|_{L^{\infty}} \leq
C(T,v_{0},b_{0}).$$
We thus complete the proof of Lemma \ref{ca2lena27}.
\end{proof}

\vskip .3in
With the above estimates of both \textbf{Case 1} and \textbf{Case 2} at our disposal, we are now ready to prove Theorem \ref{Th1}.
\begin{proof}[{Proof of Theorem \ref{Th1}}]
Applying $\Lambda^{s}$ to $\eqref{RMHDTQye07}$ and multiplying it by $\Lambda^{s}\omega$ yield
\begin{align}\label{RMHDTQye26}
\frac{1}{2}\frac{d}{dt}\|\Lambda^{s}\omega(t)\|_{L^2}^{2}=-\int_{\mathbb{R}^{2}}[\Lambda^{s},\,u\cdot\nabla]\omega \Lambda^{s}\omega\,dx +\int_{\mathbb{R}^{2}}\Lambda^{s}\nabla^{\perp}\nabla\cdot(b\otimes b)\Lambda^{s}\omega.
\end{align}
Applying $\Lambda^{s+2-\beta}$ to $\eqref{eqrtyett01}_{2}$ and multiplying it by $\Lambda^{s+2-\beta}b$, we get
\begin{align}\label{RMHDTQye27}
\frac{1}{2}\frac{d}{dt}\|\Lambda^{s+2-\beta}b(t)\|_{L^2}^{2}
+\|\Lambda^{s+2}b\|_{L^2}^{2}=&\int_{\mathbb{R}^{2}}{\Lambda^{s+2-\beta}(b\cdot \nabla u) \cdot \Lambda^{s+2-\beta}b\,dx} \nonumber\\&-\int_{\mathbb{R}^{2}}{\Lambda^{s+2-\beta}(u\cdot \nabla b) \cdot \Lambda^{s+2-\beta}b\,dx}.
\end{align}
Due to $\nabla\cdot u=0$, we have by \eqref{yfzdf68b} that
\begin{align}\label{RMHDTQye28}
\int_{\mathbb{R}^{2}}[\Lambda^{s},\,u\cdot\nabla]\omega \Lambda^{s}\omega\,dx &=
\int_{\mathbb{R}^{2}}[\Lambda^{s}\partial_{i},\,u_{i}]\omega \Lambda^{s}\omega\,dx\nonumber\\& \leq C\|[\Lambda^{s}\partial_{i},\,u_{i}]\omega \|_{L^{2}}\|\Lambda^{s}\omega\|_{L^{2}}\nonumber\\& \leq C(\|\nabla u\|_{L^{\infty}}\|\Lambda^{s}\omega\|_{L^{2}}+\|\omega\|_{L^{\infty}}\|\Lambda^{s+1}
u\|_{L^{2}})\|\Lambda^{s}\omega\|_{L^{2}}
\nonumber\\& \leq C(\|\nabla u\|_{L^{\infty}}\|\Lambda^{s}\omega\|_{L^{2}}+\|\omega\|_{L^{\infty}}\|\Lambda^{s}
\omega\|_{L^{2}})\|\Lambda^{s}\omega\|_{L^{2}}\nonumber\\& \leq C(\|\omega\|_{L^{2}}+\|\omega\|_{L^{\infty}})\|\Lambda^{s}\omega\|_{L^{2}}^{2},
\end{align}
where we have applied the facts due to $v=u+(-\Delta)^{\gamma}u$ with $\gamma>0$
\begin{align}\label{Rcdfr6yma1}
\|\nabla u\|_{L^{\infty}}\leq C(\|\omega\|_{L^{2}}+\|\omega\|_{L^{\infty}}),\qquad \|\Lambda^{s+\sigma}
u\|_{L^{2}}\leq \|\Lambda^{s}\omega\|_{L^{2}},\quad \sigma\leq 2\gamma+1.
\end{align}
The simple proof of \eqref{Rcdfr6yma1} will be given at the end of Appendix B.
In view of \eqref{yzzqq1}, it entails
\begin{align}\label{RMHDTQye29}
\int_{\mathbb{R}^{2}}\Lambda^{s}\nabla^{\perp}\nabla\cdot(b\otimes b)\Lambda^{s}\omega&\leq C\|\Lambda^{s}\nabla^{\perp}\nabla\cdot(b\otimes b)\|_{L^{2}}\|\Lambda^{s}\omega\|_{L^{2}}\nonumber\\&\leq C\|b\|_{L^{\infty}}\|\Lambda^{s+2}b\|_{L^{2}}\|\Lambda^{s}\omega\|_{L^{2}}
\nonumber\\&\leq\frac{1}{8}\|\Lambda^{s+2}b\|_{L^{2}}^{2}+ C\|b\|_{L^{\infty}}^{2} \|\Lambda^{s}\omega\|_{L^{2}}^{2}.
\end{align}
Noticing $\nabla\cdot b=0$, it follows from \eqref{yzzqq1} and \eqref{Rcdfr6yma1} that
\begin{align}\label{RMHDTQye30}&
\int_{\mathbb{R}^{2}}{\Lambda^{s+2-\beta}(b\cdot \nabla u) \cdot \Lambda^{s+2-\beta}b\,dx}
\nonumber\\&\leq C\|\Lambda^{s+2-2\beta}
(b\cdot \nabla u)\|_{L^{2}}\|\Lambda^{s+2}b\|_{L^{2}}
\nonumber\\&\leq C(\|\nabla u\|_{L^{\infty}}\|\Lambda^{s+2-2\beta}
b\|_{L^{2}}+\|b\|_{L^{\infty}}\|\Lambda^{s+3-2\beta}
u\|_{L^{2}})\|\Lambda^{s+2}b\|_{L^{2}}
\nonumber\\&\leq C(\|\nabla u\|_{L^{\infty}}\|b\|_{L^{2}}^{\frac{ \beta}{s+2-\beta}}\|\Lambda^{s+2-\beta}
b\|_{L^{2}}^{\frac{s+2-2\beta}{s+2-\beta}} +\|b\|_{L^{\infty}}\|\Lambda^{s}
\omega\|_{L^{2}})\|\Lambda^{s+2}b\|_{L^{2}}
\nonumber\\&\leq\frac{1}{8}\|\Lambda^{s+2}b\|_{L^{2}}^{2}+ C\|b\|_{L^{\infty}}^{2}\|\Lambda^{s}
\omega\|_{L^{2}}^{2}\nonumber\\&\quad +C(\|\omega\|_{L^{2}}+\|\omega\|_{L^{\infty}})^{2}
\|b\|_{L^{2}}^{\frac{2\beta}{s+2-\beta}}(1+\|\Lambda^{s+2-\beta}
b\|_{L^{2}}^{2}).
\end{align}
Since $\nabla\cdot u=0$, we also conclude
\begin{align}\label{RMHDTQye31}&
\int_{\mathbb{R}^{2}}{\Lambda^{s+2-\beta}(u\cdot \nabla b) \cdot \Lambda^{s+2-\beta}b\,dx}\nonumber\\&=\int_{\mathbb{R}^{2}}{
[\Lambda^{s+2-\beta},\,u\cdot \nabla]b\cdot \Lambda^{s+2-\beta}b\,dx}
\nonumber\\&\leq C\|[\Lambda^{s+2-\beta},\,u\cdot \nabla]b\|_{L^{2}}\|\Lambda^{s+2-\beta}b\|_{L^{2}}
\nonumber\\&\leq C(\|\nabla u\|_{L^{\infty}}\|\Lambda^{s+2-\beta}b\|_{L^{2}}+\|\nabla b\|_{L^{\infty}}\|\Lambda^{s+2-\beta}u\|_{L^{2}})\|\Lambda^{s+2-\beta}b\|_{L^{2}}
\nonumber\\&\leq C(\|\omega\|_{L^{2}}+\|\omega\|_{L^{\infty}})\|\Lambda^{s+2-\beta}b\|_{L^{2}}^{2}
+C\|\nabla b\|_{L^{\infty}}\|\Lambda^{s}\omega\|_{L^{2}} \|\Lambda^{s+2-\beta}b\|_{L^{2}}.
\end{align}
Summing up \eqref{RMHDTQye26}, \eqref{RMHDTQye27}, \eqref{RMHDTQye28}, \eqref{RMHDTQye29}, \eqref{RMHDTQye30} and \eqref{RMHDTQye31}, it directly gives
\begin{align}\label{RMHDTQye32}&
\frac{d}{dt}(\|\Lambda^{s}\omega(t)\|_{L^2}^{2}+\|\Lambda^{s+2-\beta}b(t)\|_{L^2}^{2})
+\|\Lambda^{s+2}b\|_{L^2}^{2}\nonumber\\&\leq  C(1+\|\omega\|_{L^{2}}^{2}+\|\omega\|_{L^{\infty}}^{2}+\|\nabla b\|_{L^{\infty}}+\|b\|_{L^{\infty}}^{2})(\|\Lambda^{s}\omega\|_{L^{2}}^{2}
+\|\Lambda^{s+2-\beta}b\|_{L^2}^{2}).
\end{align}
The estimates of Lemma \ref{RMHDY1}-Lemma \ref{RMHDY4} (or Lemma \ref{ca2lena25}-Lemma \ref{ca2lena27}) allow us to show
$$\int_{0}^{T}(1+\|\omega\|_{L^{2}}^{2}+\|\omega\|_{L^{\infty}}^{2}+\|\nabla b\|_{L^{\infty}}+\|b\|_{L^{\infty}}^{2})(t)\,dt\leq C(T,v_{0},b_{0}).$$
This together with the Gronwall inequality, we deduce from \eqref{RMHDTQye32} that
$$\sup_{t\in [0,\,T]}(\|\Lambda^{s}\omega(t)\|_{L^2}^{2}+\|\Lambda^{s+2-\beta}b(t)\|_{L^2}^{2})
+\int_{0}^{T}\|\Lambda^{s+2}b(t)\|_{L^2}^{2}\,dt\leq C(T,v_{0},b_{0}).$$
Consequently, this completes the proof of Theorem \ref{Th1}.
\end{proof}

\vskip .3in
\section{The proof of Theorem \ref{Th2} and  Theorem \ref{addTh2}}\setcounter{equation}{0}

We remark that when $\gamma=0$ and $\alpha=2$, the corresponding system admits a unique global regular solution, see \cite{TrYuna,Yamzaki14aml} for details. Consequently, it is sufficient to consider the case $\alpha+\gamma=2$ with $ \alpha\in[0,\,2)$.

\vskip .1in

\begin{center}
The proof of Theorem \ref{Th2}
\end{center}
\vskip .1in

In this case, we should keep in mind that $\alpha+\gamma=2$ with $ \alpha\in (0,\,2)$. Now we begin with the basic energy estimate.
\begin{lemma}\label{RMHDS31}
Assume $(v_{0},\,b_{0})$ satisfies the assumptions stated in
Theorem \ref{Th2}.
Then the corresponding solution $(v, b)$
of (\ref{logRG2DMHD}) admits the following bound
\begin{align*}
\|u(t)\|_{L^{2}}^{2}+\|\Lambda^{\gamma}\mathcal{L}u(t)\|_{L^{2}}^{2}
+\|b(t)\|_{L^{2}}^{2}&+\int_{0}^{t}{(\|\Lambda^{\alpha}\mathcal{L}u(\tau)\|_{L^{2}}^{2}+
\|\Lambda^{\alpha+\gamma}\mathcal{L}^{2}u(\tau)\|_{L^{2}}^{2})\,d\tau}\nonumber\\&\leq
C(v_{0},b_{0}).
\end{align*}
In particular, due to $\alpha+\gamma=2$, there holds true
\begin{align}\label{RMHDyet301}
\|u(t)\|_{L^{2}}^{2}+\|\Lambda^{\gamma}\mathcal{L}u(t)\|_{L^{2}}^{2}+\|b(t)\|_{L^{2}}^{2}+\int_{0}^{t}{\Big\|\frac{\Lambda^{2}}{g^{2}(\Lambda)} u(\tau)\Big\|_{L^{2}}^{2} \,d\tau} \leq
C(v_{0},b_{0}).
\end{align}
\end{lemma}

\begin{proof}
The proof can be performed as that of Lemma \ref{RMHDY1}. We thus omit the details.
\end{proof}

\vskip .1in
Our next goal is to derive the following key estimate.

\begin{lemma}\label{RMHDS32}
Assume $(v_{0},\,b_{0})$ satisfies the assumptions stated in
Theorem \ref{Th2}. Then the corresponding solution $(v, b)$
of (\ref{logRG2DMHD}) admits the following bound
\begin{align}\label{RMHDyet302}
\sup_{t\in [0,\,T]}(\|b(t)\|_{L^{\infty}}
+\|\nabla b(t)\|_{L^{2}}
+\|v(t)\|_{L^{2}})
+\int_{0}^{T}{\|\Lambda^{\alpha}\mathcal{L}v(\tau)\|_{L^{2}}^{2} \,d\tau}\leq
C(T,v_{0},b_{0}).
\end{align}
In particular, it holds
\begin{align}\label{sdfrD21rd}
 \int_{0}^{T}{(\|\nabla u(\tau)\|_{L^{\infty}}^{2}+\|\Delta u(\tau)\|_{L^{2}}^{2}) \,d\tau}\leq
C(T,v_{0},b_{0}).
\end{align}
\end{lemma}

\begin{proof}
We recall the second equation of \eqref{logRG2DMHD}, namely,
\begin{align}\label{RMHDyet303}
\partial_{t}b+(u\cdot\nabla)b=(b\cdot\nabla)u.
\end{align}
Multiplying \eqref{RMHDyet303} by $|b|^{q-2}b$ and integrating it over whole space, it implies
\begin{align}
\frac{1}{q}\frac{d}{dt}\|b(t)\|_{L^{q}}^{q}\leq \|\nabla u\|_{L^{\infty}}\|b\|_{L^{q}}^{q},\nonumber
\end{align}
which leads to
\begin{align}
\frac{d}{dt}\|b(t)\|_{L^{q}}\leq \|\nabla u\|_{L^{\infty}}\|b\|_{L^{q}}.\nonumber
\end{align}
By letting $q\rightarrow\infty$, we deduce
\begin{align*}
\frac{d}{dt}\|b(t)\|_{L^{\infty}}\leq C\|\nabla u\|_{L^{\infty}}\|b\|_{L^{\infty}},
\end{align*}
which implies
\begin{align}\label{RMHDyet304}
\frac{d}{dt}\|b(t)\|_{L^{\infty}}^{m}\leq C(m)\|\nabla u\|_{L^{\infty}}\|b\|_{L^{\infty}}^{m},\quad \forall\,m\in[2,\,\infty).
\end{align}
Applying $\nabla$ to \eqref{RMHDyet303}, we get
\begin{align}\label{RMHDyet305}
\partial_{t}\nabla b+(u\cdot\nabla)\nabla b=\nabla b\cdot\nabla u+b\cdot\nabla\nabla u-\nabla u\cdot\nabla b.
\end{align}
Taking the $L^2$-inner product of \eqref{RMHDyet305} with $\nabla b$, one arrives at
\begin{align*}
\frac{d}{dt}\|\nabla b(t)\|_{L^{2}}^{2}&\leq C\|\nabla u\|_{L^{\infty}}\|\nabla b\|_{L^{2}}^{2}+C\|b\|_{L^{\infty}}\|\Delta u\|_{L^{2}}\|\nabla b\|_{L^{2}}\nonumber\\&\leq C(\|\nabla u\|_{L^{\infty}}+\|\Delta u\|_{L^{2}})(\|\nabla b\|_{L^{2}}^{2}+\|b\|_{L^{\infty}}^{2}),
\end{align*}
which allows us to show
\begin{align}\label{RMHDyet306}
\frac{d}{dt}\|\nabla b(t)\|_{L^{2}}^{3} \leq C(\|\nabla u\|_{L^{\infty}}+\|\Delta u\|_{L^{2}})(\|\nabla b\|_{L^{2}}^{3}+\|b\|_{L^{\infty}}^{3}).
\end{align}
Taking the $L^2$-inner product of $\eqref{logRG2DMHD}_{1}$ with $v$, it gives that
\begin{align*}
\frac{1}{2}\frac{d}{dt}\|v(t)\|_{L^{2}}^{2}+\|\Lambda^{\alpha}\mathcal{L}v\|_{L^{2}}^{2}&=-\int_{\mathbb{R}^{2}}{\Big(\sum_{j=1}^{2}v_{j}\nabla u_{j}\Big) \cdot v\,dx}+\int_{\mathbb{R}^{2}}{(b\cdot \nabla b) \cdot v\,dx}\nonumber\\&\leq C\|\nabla u\|_{L^{\infty}}\|v\|_{L^{2}}^{2}+C\|b\|_{L^{\infty}}\|\nabla b\|_{L^{2}}\|v\|_{L^{2}}
\nonumber\\&\leq C\|\nabla u\|_{L^{\infty}}\|v\|_{L^{2}}^{2}+C\|v\|_{L^{2}}^{2} +C\|\nabla b\|_{L^{2}}^{3}
+C\|b\|_{L^{\infty}}^{6}.
\end{align*}
This allows us to deduce
\begin{align}\label{RMHDyet307}
\frac{d}{dt}\|v(t)\|_{L^{2}}^{2}+\|\Lambda^{\alpha}\mathcal{L}v\|_{L^{2}}^{2}\leq C(1+\|\nabla u\|_{L^{\infty}})\|v\|_{L^{2}}^{2}+C\|\nabla b\|_{L^{2}}^{3}
+C\|b\|_{L^{\infty}}^{6}.
\end{align}
Taking $m=3$ and $m=6$ in \eqref{RMHDyet304}, we then get by summing up \eqref{RMHDyet304}, \eqref{RMHDyet306} and \eqref{RMHDyet307}
\begin{align}\label{RMHDzxr21}
\frac{d}{dt}V(t)+H(t)\leq C(1+\|\nabla u\|_{L^{\infty}}+\|\Delta u\|_{L^{2}})V(t),
\end{align}
where
$$V(t):=\|b(t)\|_{L^{\infty}}^{3}+\|b(t)\|_{L^{\infty}}^{6}+
\|\nabla b(t)\|_{L^{2}}^{3}+\|v(t)\|_{L^{2}}^{2},\quad H(t):=\|\mathcal{L}\Lambda^{\alpha}v(t)\|_{L^{2}}^{2}.$$
According to the assumptions on $g$ (more precisely, $g$ grows logarithmically), one may conclude that for any fixed $\delta>0$, there exists $N=N(\delta)$ satisfying
$$g(r)\leq \widetilde{C}r^{\delta},\quad \forall\,r\geq N$$
with the constant $\widetilde{C}=\widetilde{C}(\delta)$. With this observation, one has for any $\delta>0$
\begin{align}\label{logtt002}
\|\Lambda^{\kappa}\mathcal{L}u\|_{L^{2}}^{2}&=
\int_{|\xi|< N(\delta)}{ \frac{|\xi|^{2\kappa}}{g^{2}(|\xi|)}|\widehat{u}(\xi)|^{2}\,d\xi}
+\int_{|\xi|\geq N(\delta)}{ \frac{|\xi|^{2\kappa}}{g^{2}(|\xi|)}|\widehat{u}(\xi)|^{2}\,d\xi}\nonumber\\
&\geq
\int_{|\xi|\geq N(\delta)}{ \frac{|\xi|^{2\kappa}}{\big[\widetilde{C}|\xi|^{\delta}
\big]^{2}}|\widehat{u}(\xi)|^{2}\,d\xi}\nonumber\\
&=
\int_{\mathbb{R}^{2}}{ \frac{|\xi|^{2\kappa}}{\big[\widetilde{C}|\xi|^{\delta}
\big]^{2}}|\widehat{u}(\xi)|^{2}\,d\xi}-\int_{|\xi|<N(\delta)}{ \frac{|\xi|^{2\kappa}}{\big[\widetilde{C}|\xi|^{\delta}
\big]^{2}}|\widehat{u}(\xi)|^{2}\,d\xi}
\nonumber\\
&\geq  C_{0}\|\Lambda^{\kappa-\delta}u\|_{L^{2}}^{2}-
\widetilde{C_{0}}\|u\|_{L^{2}}^{2},
\end{align}
where $C_{0}$ and $\widetilde{C_{0}}$ depend only on $\delta$.
Due to $v=u+(-\Delta)^{\gamma}\mathcal{L}^{2}u$, we may deduce by the argument used in proving \eqref{logtt002} that
\begin{align}\label{RMHDzxr2011}
\|\Lambda^{r_{1}}u\|_{L^{2}}\leq C(\|u\|_{L^{2}}+\|\mathcal{L}\Lambda^{\alpha}v\|_{L^{2}})
\end{align}
for any $r_{1}<\alpha+2\gamma\equiv 2+\gamma>2$.
Invoking the high-low frequency technique gives
\begin{eqnarray}
\|\nabla u\|_{L^{\infty}}\leq \|\Delta_{-1}\nabla u\|_{L^{\infty}}+\sum_{l=0}^{N-1}\|\Delta_{l}\nabla u\|_{L^{\infty}}+\sum_{l=N}^{\infty}\|\Delta_{l}\nabla u\|_{L^{\infty}},\nonumber
\end{eqnarray}
where $\Delta_{l}$ ($l=-1,0,1,\cdot\cdot\cdot$) denote the frequency
operator (see Appendix for details).
One obtains that by using the Bernstein inequality (see Lemma \ref{vfgty8xc})
$$
\|\Delta_{-1}\nabla u\|_{L^{\infty}}\leq C\|u\|_{L^{2}},$$
$$
\sum_{l=N}^{\infty}\|\Delta_{l}\nabla u\|_{L^{\infty}} \leq
 C \sum_{l=N}^{\infty}2^{2l}\|\Delta_{l}u\|_{L^{2}} =
C \sum_{l=N}^{\infty}2^{l(2-r_{1})}\|\Delta_{l}\Lambda^{r_{1}}u\|_{L^{2}} \leq  C2^{N(2-r_{1})}\|\Lambda^{r_{1}}u\|_{L^{2}},
$$
where $r_{1}>2$. By the Bernstein inequality again and the Plancherel theorem, one gets
\begin{align}
\sum_{l=0}^{N-1}\|\Delta_{l}\nabla u\|_{L^{\infty}} \leq C &\sum_{l=0}^{N-1}2^{2l}\|\Delta_{l} u\|_{L^{2}}\leq
C \sum_{l=0}^{N-1} \|\Delta_{l} \Lambda^{2}u\|_{L^{2}}\nonumber\\
 \leq &
C \sum_{l=0}^{N-1} \|\varphi(2^{-l}\xi) |\xi|^{2}\widehat{u}(\xi)\|_{L^{2}}\nonumber\\
 =&C\sum_{l=0}^{N-1} \Big\|\varphi(2^{-l}\xi)
g^{2}(|\xi|)
\frac{|\xi|^{2}}{ g^{2}(|\xi|) }\widehat{u}(\xi)\Big\|_{L^{2}}
\nonumber\\
 \leq &C\sum_{l=0}^{N-1} g^{2}(2^{l}) \Big\|\frac{|\xi|^{2}}{g^{2}(|\xi|) }\widehat{\Delta_{l}u}(\xi)\Big\|_{L^{2}}
\nonumber\\
 \leq &C\Big(\sum_{l=0}^{N-1} g^{4}(2^{l})\Big)^{\frac{1}{2}}\left(\sum_{l=0}^{N-1}
\Big\|\frac{|\xi|^{2}}{ g^{2}(|\xi|)}\widehat{\Delta_{l}u}(\xi)\Big\|_{L^{2}}^{2}\right)^{\frac{1}{2}}
\nonumber\\
\leq &Cg^{2}(2^{N})\Big(\sum_{l=1}^{N-1} 1\Big)^{\frac{1}{2}} \Big\|\frac{\Lambda^{2}}{ g^{2}(\Lambda)}u \Big\|_{L^{2}}
\nonumber\\
\leq &Cg^{2}(2^{N})\sqrt{N} \Big\|\frac{\Lambda^{2}}{ g^{2}(\Lambda)}u \Big\|_{L^{2}},\nonumber
\end{align}
where we have used the fact that $g$ is a non-decreasing function. Taking $2<r_{1}<2+\gamma$, we therefore obtain by summarizing the above three estimates
\begin{eqnarray}
\|\nabla u\|_{L^{\infty}}\leq C\|u\|_{L^{2}}+Cg^{2}(2^{N})\sqrt{N} \Big\|\frac{\Lambda^{2}}{ g^{2}(\Lambda)}u \Big\|_{L^{2}}+C2^{N(2-r_{1})}\|\mathcal{L}\Lambda^{\alpha}v\|_{L^{2}}.\nonumber
\end{eqnarray}
By the same argument, we have
\begin{eqnarray}
\|\Delta u\|_{L^{2}}\leq C\|u\|_{L^{2}}+Cg^{2}(2^{N})\sqrt{N} \Big\|\frac{\Lambda^{2}}{ g^{2}(\Lambda)}u \Big\|_{L^{2}}+C2^{N(2-r_{1})}\|\mathcal{L}\Lambda^{\alpha}v\|_{L^{2}}.\nonumber
\end{eqnarray}
It follows from \eqref{RMHDzxr21} and \eqref{RMHDzxr2011} that
\begin{eqnarray}
 \frac{d}{dt}V(t)+H(t)
 \leq    CV(t)+Cg^{2}(2^{N})\sqrt{N} \Big\|\frac{\Lambda^{2}}{ g^{2}(\Lambda)}u \Big\|_{L^{2}}V(t)+C2^{N(2-r_{1})}H^{\frac{1}{2}}(t)V(t).\nonumber
\end{eqnarray}
Taking $N$ such
$$2^{N}\approx \left(e+V(t)\right)^{\frac{1}{2(r_{1}-2)}},$$
it yields
\begin{align}
 \frac{d}{dt}V(t)+H(t)
  \leq&    C\left(1+\Big\|\frac{\Lambda^{2}}{ g^{2}(\Lambda)}u \Big\|_{L^{2}}\right)g^{2}\big[\left(e+V(t)\right)^{\frac{1}{2(r_{1}-2)}}\big]
  \sqrt{\ln\big(e+V(t)\big)} \big(e+V(t)\big)\nonumber\\& +CH^{\frac{1}{2}}(t)\big(e+V(t)\big)^{\frac{1}{2}}
\nonumber\\
 \leq&    C\left(1+\Big\|\frac{\Lambda^{2}}{ g^{2}(\Lambda)}u \Big\|_{L^{2}}\right)g^{2}\big[\left(e+V(t)\right)^{\frac{1}{2(r_{1}-2)}}\big]
 \sqrt{\ln\big(e+V(t)\big)} \big(e+V(t)\big)\nonumber\\ &+\frac{1}{2}H(t)+C\big(e+V(t)\big).\nonumber
\end{align}
Using the following fact
$$g^{2}\big[\left(e+V(t)\right)^{\frac{1}{2(r_{1}-2)}}\big]
 \sqrt{\ln\big(e+V(t)\big)} \big(e+V(t)\big)\geq1,$$
we thus have
\begin{align}
 \frac{d}{dt}V(t)+H(t)
 \leq C\left(1+\Big\|\frac{\Lambda^{2}}{ g^{2}(\Lambda)}u \Big\|_{L^{2}}\right)g^{2}\big[\left(e+V(t)\right)^{\frac{1}{2(r_{1}-2)}}\big]
 \sqrt{\ln\big(e+V(t)\big)} \big(e+V(t)\big).\nonumber
\end{align}
As a result, one deduces from the above inequality that
$$\int_{e+V(0)}^{e+V(t)}\frac{d\tau}{\tau\sqrt{\ln \tau} g^{2}(\tau^{\frac{1}{2(r_{1}-2)}})}\leq C
\int_{0}^{t}{
\left(1+\Big\|\frac{\Lambda^{2}}{ g^{2}(\Lambda)}u (\tau) \Big\|_{L^{2}}\right)\,d\tau}.$$
Making use of the following facts due to \eqref{logcobd} and \eqref{RMHDyet301}, respectively
$$\int_{e}^{\infty}\frac{d\tau}{\tau\sqrt{\ln \tau} g^{2}(\tau^{\frac{1}{2(r_{1}-2)}})}=\sqrt{2(r_{1}-2)}\int_{e^{\frac{1}{2(r_{1}-2)}}}^{\infty}\frac{d\tau}{\tau\sqrt{\ln \tau} g^{2}(\tau)}=\infty$$
and
$$\int_{0}^{T}{
\left(1+\Big\|\frac{\Lambda^{2}}{ g^{2}(\Lambda)}u (\tau) \Big\|_{L^{2}}\right)\,d\tau}\leq C(T,\,u_{0},\,w_{0}),$$
we have
$$\sup_{t\in [0,\,T]}V(t)+\int_{0}^{T}H(\tau)\,d\tau\leq C(T,v_{0},b_{0}),$$
which is nothing but \eqref{RMHDyet302}.
The desired \eqref{sdfrD21rd} follows from \eqref{RMHDzxr2011}.
We therefore finish the proof of Lemma \ref{RMHDS32}.
\end{proof}

\vskip .1in
Now we are going to improve the regularity of $b$, which will be used to derive \eqref{RMHDyet313}.

\begin{lemma}\label{zxdDSma}
Assume $(v_{0},\,b_{0})$ satisfies the assumptions stated in
Theorem \ref{Th2}. Then the corresponding solution $(v, b)$
of (\ref{logRG2DMHD}) admits the following bound for any $0<\nu<\gamma$
\begin{align}\label{ftyu45t1}
\sup_{t\in [0,\,T]}\|\Lambda^{1+\nu}b(t)\|_{L^{2}}\leq
C(T,v_{0},b_{0}).
\end{align}
\end{lemma}

\begin{proof}
Applying $\Lambda^{1+\nu}$ to \eqref{RMHDyet303} and multiplying it by $\Lambda^{1+\nu}b$, we obtain
\begin{align*}
\frac{1}{2}\frac{d}{dt}\|\Lambda^{1+\nu}b(t)\|_{L^{2}}^{2}
&=-\int_{\mathbb{R}^{2}}{[\Lambda^{1+\nu}, u\cdot \nabla ]b \cdot \Lambda^{1+\nu} b\,dx}+\int_{\mathbb{R}^{2}}{\Lambda^{1+\nu}(b\cdot \nabla u) \cdot \Lambda^{1+\nu} b\,dx}.
\end{align*}
Thanks to $\nabla\cdot u=0$, we get by \eqref{yfzdf68b} that
\begin{align*}
\int_{\mathbb{R}^{2}}{[\Lambda^{1+\nu}, u\cdot \nabla ]b \cdot \Lambda^{1+\nu} b\,dx}&= \int_{\mathbb{R}^{2}}{[\Lambda^{1+\nu}\partial_{i}, u_{i}]b_{j}  \Lambda^{1+\nu} b_{j}\,dx}\nonumber\\&\leq
C\|[\Lambda^{1+\nu}\partial_{i}, u_{i}]b_{j}\|_{L^{2}}\|\Lambda^{1+\nu} b\|_{L^{2}}
\nonumber\\&\leq C(\|\nabla u\|_{L^{\infty}}\|\Lambda^{1+\nu} b\|_{L^{2}}+\| b\|_{L^{\infty}}\|\Lambda^{2+\nu} u\|_{L^{2}})\|\Lambda^{1+\nu} b\|_{L^{2}}\nonumber\\& \leq
C\|\nabla u\|_{L^{\infty}}\|\Lambda^{1+\nu} b\|_{L^{2}}^{2}\nonumber\\&\quad +C\| b\|_{L^{\infty}}(\|u\|_{L^{2}}+\|\mathcal{L}\Lambda^{\alpha}v\|_{L^{2}})
\|\Lambda^{1+\nu} b\|_{L^{2}},
\end{align*}
where we have used the following fact for any $0<\nu<\gamma$ (see \eqref{RMHDzxr2011} for details)
$$\|\Lambda^{2+\nu} u\|_{L^{2}}\leq C(\|u\|_{L^{2}}+\|\mathcal{L}\Lambda^{\alpha}v\|_{L^{2}}).$$
In view of \eqref{yzzqq1} and the above fact, it yields
\begin{align*}
\int_{\mathbb{R}^{2}}{\Lambda^{1+\nu}(b\cdot \nabla u) \cdot \Lambda^{1+\nu} b\,dx}&\leq
C\|\Lambda^{1+\nu}(b\cdot \nabla u)\|_{L^{2}}\|\Lambda^{1+\nu} b\|_{L^{2}}
\nonumber\\&\leq C(\|\nabla u\|_{L^{\infty}}\|\Lambda^{1+\nu} b\|_{L^{2}}+\| b\|_{L^{\infty}}\|\Lambda^{2+\nu} u\|_{L^{2}})\|\Lambda^{1+\nu} b\|_{L^{2}}
\nonumber\\& \leq
C\|\nabla u\|_{L^{\infty}}\|\Lambda^{1+\nu} b\|_{L^{2}}^{2}\nonumber\\&\quad +C\| b\|_{L^{\infty}}(\|u\|_{L^{2}}+\|\mathcal{L}\Lambda^{\alpha}v\|_{L^{2}})
\|\Lambda^{1+\nu} b\|_{L^{2}}.
\end{align*}
Consequently, we deduce
\begin{align*}
\frac{d}{dt}\|\Lambda^{1+\nu}b(t)\|_{L^{2}}^{2}\leq
C\|\nabla u\|_{L^{\infty}}\|\Lambda^{1+\nu} b\|_{L^{2}}^{2} +C\| b\|_{L^{\infty}}(\|u\|_{L^{2}}+\|\mathcal{L}\Lambda^{\alpha}v\|_{L^{2}})
\|\Lambda^{1+\nu} b\|_{L^{2}}.
\end{align*}
By the Gronwall inequality, we conclude by using \eqref{sdfrD21rd} and \eqref{RMHDyet302} that
$$\sup_{t\in [0,\,T]}\|\Lambda^{1+\nu}b(t)\|_{L^{2}}\leq
C(T,v_{0},b_{0}).$$
This ends the proof of Lemma \ref{zxdDSma}.
\end{proof}

\vskip .1in
The following lemma concerns the higher regularity of $v$.

\begin{lemma}\label{RMHDS33}
Assume $(v_{0},\,b_{0})$ satisfies the assumptions stated in
Theorem \ref{Th2}. Then the corresponding solution $(v, b)$
of (\ref{logRG2DMHD}) admits the following bound
\begin{align}\label{RMHDyet313}
 \sup_{t\in [0,\,T]}\|\Lambda^{\alpha}v(t)\|_{L^{2}}
+\int_{0}^{T}{\|\Lambda^{2\alpha}\mathcal{L}v(\tau)\|_{L^{2}}^{2} \,d\tau}\leq
C(T,v_{0},b_{0}).
\end{align}
Moreover, we have
\begin{align}\label{RMHDyet315}
 \sup_{t\in [0,\,T]}\|\nabla b(t)\|_{L^{\infty}}\leq C(T,v_{0},b_{0}).
\end{align}
\end{lemma}

\begin{proof}
Multiplying $\eqref{RG2DMHD}_{1}$ by $\Lambda^{2\alpha}v$ and integrating it over $\mathbb{R}^{2}$, we get
\begin{align*}
\frac{1}{2}\frac{d}{dt}\|\Lambda^{\alpha}v(t)\|_{L^{2}}^{2}
+\|\Lambda^{2\alpha}\mathcal{L}v\|_{L^{2}}^{2}&=-\int_{\mathbb{R}^{2}}{\Big(\sum_{j=1}^{2}v_{j}\nabla u_{j}\Big) \cdot \Lambda^{2\alpha}v\,dx}+\int_{\mathbb{R}^{2}}{(b\cdot \nabla b) \cdot \Lambda^{2\alpha}v\,dx}\nonumber\\&\quad - \int_{\mathbb{R}^{2}}{[\Lambda^{\alpha},\,u\cdot \nabla]v \cdot \Lambda^{\alpha}v\,dx}.
\end{align*}
Thanks to \eqref{logtt002}, we have for any $r_{3}<2\alpha$
\begin{align}\label{fgy7ubvc1}
\|\Lambda^{r_{3}}v\|_{L^{2}}\leq C(\|v\|_{L^{2}}+\|\Lambda^{2\alpha}\mathcal{L}v\|_{L^{2}}).
\end{align}
Taking $0<\delta<\min\{1,\,\alpha\}$ and using \eqref{fgy7ubvc1}, the first term can be bounded by
\begin{align*}
\int_{\mathbb{R}^{2}}{\Big(\sum_{j=1}^{2}v_{j}\nabla u_{j}\Big) \cdot \Lambda^{2\alpha}v\,dx}&\leq C\|\Lambda^{\delta}(v\nabla u)\|_{L^{2}}\|\Lambda^{2\alpha-\delta}v\|_{L^{2}}\nonumber\\&\leq
C(\|\nabla u\|_{L^{\infty}}\|\Lambda^{\delta}v\|_{L^{2}}+\|\Lambda^{\delta}\nabla u\|_{L^{\frac{2}{\delta}}}\|v\|_{L^{\frac{2}{1-\delta}}})\|\Lambda^{2\alpha-\delta}
v\|_{L^{2}}\nonumber\\&\leq
C(\|\nabla u\|_{L^{\infty}}+\|\Delta u\|_{L^{2}})
\|v\|_{L^{2}}^{\frac{\alpha-\delta}{\alpha}}\|\Lambda^{\alpha}
v\|_{L^{2}}^{\frac{\delta}{\alpha}}\nonumber\\&\quad\times
(\|v\|_{L^{2}}+\|\Lambda^{2\alpha}\mathcal{L}v\|_{L^{2}})\nonumber\\&\leq
\frac{1}{8}\|\Lambda^{2\alpha}\mathcal{L}v\|_{L^{2}}^{2}+C(1+\|\nabla u\|_{L^{\infty}}^{2}+\|\Delta u\|_{L^{2}}^{2})(1+\|\Lambda^{\alpha}
v\|_{L^{2}}^{2}).
\end{align*}
On the other hand, one may obtain
\begin{align*}\int_{\mathbb{R}^{2}}{(b\cdot \nabla b) \cdot \Lambda^{2\alpha}v\,dx}&\leq C\|\Lambda^{1+\nu}(bb)\|_{L^{2}}\|\Lambda^{2\alpha-\nu}v\|_{L^{2}}\nonumber\\&\leq
C\|b\|_{L^{\infty}}\|\Lambda^{1+\nu}b\|_{L^{2}}
(\|v\|_{L^{2}}+\|\Lambda^{2\alpha}\mathcal{L}v\|_{L^{2}})\nonumber\\&\leq
\frac{1}{8}\|\Lambda^{2\alpha}\mathcal{L}v\|_{L^{2}}^{2}+C\|b\|_{L^{\infty}}^{2}
\|\Lambda^{1+\nu}b\|_{L^{2}}^{2}+C\|b\|_{L^{\infty}} \|v\|_{L^{2}}
\|\Lambda^{1+\nu}b\|_{L^{2}},
\end{align*}
where $\nu$ is given by \eqref{ftyu45t1}.
For the last term, due to $\nabla\cdot u=0$, we have by \eqref{yfzdf68b} that
\begin{align*}
- \int_{\mathbb{R}^{2}}{[\Lambda^{\alpha},\,u\cdot \nabla]v \cdot \Lambda^{\alpha}v\,dx} &=-
\int_{\mathbb{R}^{2}}[\Lambda^{\alpha}\partial_{i},\,u_{i}]v_{j} \Lambda^{\alpha}v_{j}\,dx\nonumber\\& \leq C\|[\Lambda^{\alpha}\partial_{i},\,u_{i}]v_{j} \|_{L^{2}}\|\Lambda^{\alpha}v\|_{L^{2}}\nonumber\\& \leq C(\|\nabla u\|_{L^{\infty}}\|\Lambda^{\alpha}v\|_{L^{2}}+\|v\|_{L^{p_{1}}}\|\Lambda^{\alpha+1}
u\|_{L^{\frac{2p_{1}}{p_{1}-2}}})\|\Lambda^{\alpha}v\|_{L^{2}}
\nonumber\\& \leq C(\|v\|_{
L^{2}}+\|\Lambda^{\alpha}\mathcal{L}v\|_{
L^{2}})\|\Lambda^{2\alpha}
\mathcal{L}v\|_{L^{2}} \|\Lambda^{\alpha}v\|_{L^{2}}
\nonumber\\&\quad +C\|\nabla u\|_{L^{\infty}}\|\Lambda^{\alpha}v\|_{L^{2}}^{2}
\nonumber\\& \leq \frac{1}{8}\|\Lambda^{2\alpha}\mathcal{L}v\|_{L^{2}}^{2}+C(\|\nabla u\|_{L^{\infty}}+\|v\|_{L^{2}}^{2}+\|\Lambda^{\alpha}\mathcal{L}v\|_{
L^{2}}^{2})\|\Lambda^{\alpha}v\|_{L^{2}}^{2},
\end{align*}
where $p_{1}>2$ satisfies
$$\frac{1-\alpha}{2}<\frac{1}{p_{1}}< \frac{\alpha+2\gamma-1}{2}.$$
Combining all the above estimates implies that
\begin{align*}
 \frac{d}{dt}\|\Lambda^{\alpha}v(t)\|_{L^{2}}^{2}
+\|\Lambda^{2\alpha}\mathcal{L}v\|_{L^{2}}^{2}&\leq C(1+\|\nabla u\|_{L^{\infty}}^{2}+\|\Delta u\|_{L^{2}}^{2}+\|\Lambda^{\alpha}\mathcal{L}v\|_{
L^{2}}^{2})(1+\|\Lambda^{\alpha}
v\|_{L^{2}}^{2})\nonumber\\&\quad +C\|b\|_{L^{\infty}}^{2}
\|\Lambda^{1+\nu}b\|_{L^{2}}^{2}+C\|b\|_{L^{\infty}} \|v\|_{L^{2}}
\|\Lambda^{1+\nu}b\|_{L^{2}}.
\end{align*}
Thanks to \eqref{RMHDyet302}, \eqref{sdfrD21rd} and \eqref{ftyu45t1}, we have by using the Gronwall inequality
$$  \sup_{t\in [0,\,T]}\|\Lambda^{\alpha}v(t)\|_{L^{2}}^{2}
+\int_{0}^{T}{\|\Lambda^{2\alpha}\mathcal{L}v(\tau)\|_{L^{2}}^{2} \,d\tau}\leq
C(T,v_{0},b_{0}).$$
It follows from \eqref{RMHDyet305} and \eqref{fgy7ubvc1} that
\begin{align}\label{ertcxp316}
\frac{d}{dt}\|\nabla b(t)\|_{L^{\infty}} &\leq C\|\nabla u\|_{L^{\infty}}\|\nabla b\|_{L^{\infty}} +C\|b\|_{L^{\infty}}\|\nabla^{2} u\|_{L^{\infty}}\|\nabla b\|_{L^{\infty}} \nonumber\\&\leq C\|\nabla u\|_{L^{\infty}}\|\nabla b\|_{L^{\infty}} +C\|b\|_{L^{\infty}}
(\|u\|_{L^{2}}+\|\Lambda^{2\alpha}\mathcal{L}v\|_{L^{2}})\|\nabla b\|_{L^{\infty}}.
\end{align}
Recalling \eqref{RMHDyet302} and \eqref{sdfrD21rd}, one deduces by the Gronwall inequality
$$\sup_{t\in [0,\,T]}\|\nabla b(t)\|_{L^{\infty}}\leq C(T,v_{0},b_{0}).$$
Consequently, the proof of Lemma \ref{RMHDS33} is completed.
\end{proof}

\vskip .1in
The following lemma concerns the $L^2$-estimate of $\omega$.
\begin{lemma}\label{RMHDS34}
Assume $(v_{0},\,b_{0})$ satisfies the assumptions stated in
Theorem \ref{Th2}. Then the corresponding solution $(v, b)$
of (\ref{logRG2DMHD}) admits the following bound
\begin{align}\label{Rfdty78i17}
\sup_{t\in [0,\,T]}\|\omega(t)\|_{L^{2}}^{2}
+\int_{0}^{T}{\|\Lambda^{\alpha}\mathcal{L}\omega(\tau)\|_{L^{2}}^{2} \,d\tau}\leq
C(T,v_{0},b_{0}).
\end{align}
\end{lemma}

\begin{proof}
We first claim
\begin{align}\label{RMHDyet317}
\sup_{t\in [0,\,T]}\|\Delta b(t)\|_{L^{2}}\leq
C(T,v_{0},b_{0}).
\end{align}
To this end, applying $\Delta$ to \eqref{RMHDyet303} and multiplying the resultant by $\Delta b$, it leads to
\begin{align*}
\frac{1}{2}\frac{d}{dt}\|\Delta b(t)\|_{L^{2}}^{2}&=-\int_{\mathbb{R}^{2}}{[\Delta, u\cdot \nabla ]b \cdot \Delta b\,dx}+\int_{\mathbb{R}^{2}}{\Delta(b\cdot \nabla u) \cdot \Delta b\,dx}
\nonumber\\&\leq C\|[\Delta, u\cdot \nabla ]b\|_{L^{2}}\|\Delta b\|_{L^{2}}+C\|\Delta(b\cdot \nabla u)\|_{L^{2}}\|\Delta b\|_{L^{2}}
\nonumber\\&\leq C(\|\nabla u\|_{L^{\infty}}\|\Delta b\|_{L^{2}}+\|\nabla b\|_{L^{\infty}}\|\Delta u\|_{L^{2}})\|\Delta b\|_{L^{2}}\nonumber\\&\quad+C(\|\nabla u\|_{L^{\infty}}\|\Delta b\|_{L^{2}}+\|b\|_{L^{\infty}}\|\Lambda^{3} u\|_{L^{2}})\|\Delta b\|_{L^{2}},
\end{align*}
which implies
\begin{align*}
 \frac{d}{dt}\|\Delta b(t)\|_{L^{2}} &\leq C\|\nabla u\|_{L^{\infty}}\|\Delta b\|_{L^{2}}+C(\|\nabla b\|_{L^{\infty}}+\|b\|_{L^{\infty}})
 (\|u\|_{L^{2}}+\|\Lambda^{2\alpha}\mathcal{L}v\|_{L^{2}}).
\end{align*}
By means of \eqref{RMHDyet302}, \eqref{sdfrD21rd}, \eqref{RMHDyet315} and the Gronwall inequality, we thus obtain
\begin{align*}
\sup_{t\in [0,\,T]}\|\Delta b(t)\|_{L^{2}}\leq
C(T,v_{0},b_{0}),
\end{align*}
which is \eqref{RMHDyet317}. We remark that in this case, the vorticity $\omega:=\nabla^{\perp}\cdot v$ reads
\begin{align}\label{RMHDyet318}
\partial_{t}\omega+(u\cdot\nabla)\omega+\Lambda^{2\alpha}\mathcal{L}^{2}\omega=\nabla^{\perp}\nabla\cdot(b\otimes b).\end{align}
Taking the $L^2$ inner product of \eqref{RMHDyet318} with $\omega$ yields
\begin{align*}
\frac{1}{2}\frac{d}{dt}\|\omega(t)\|_{L^2}^{2}+\|\Lambda^{\alpha}\mathcal{L}\omega\|_{L^2}^{2}
&\leq
C\|\nabla^{\perp}\nabla\cdot(b\otimes b)\|_{L^2}\|\omega\|_{L^2}\nonumber\\
&\leq
C\|b\|_{L^{\infty}}\|\Delta b\|_{L^{2}}\|\omega\|_{L^2}.
\end{align*}
The Gronwall inequality and \eqref{RMHDyet317} allow us to get
$$\sup_{t\in [0,\,T]}\|\omega(t)\|_{L^{2}}^{2}
+\int_{0}^{T}{\|\Lambda^{\alpha}\mathcal{L}\omega(\tau)\|_{L^{2}}^{2} \,d\tau}\leq
C(T,v_{0},b_{0}).$$
Consequently, we complete the proof of Lemma \ref{RMHDS34}.
\end{proof}

\vskip .1in
With the above estimates at our disposal, we will prove Theorem \ref{Th2}.
\begin{proof}[{Proof of Theorem \ref{Th2}}]
Applying $\Lambda^{s}$ to $\eqref{RMHDyet318}$ and multiplying it by $\Lambda^{s}\omega$ yield
\begin{align*}
\frac{1}{2}\frac{d}{dt}\|\Lambda^{s}\omega(t)\|_{L^2}^{2}+\|\Lambda^{s+\alpha}
\mathcal{L}\omega\|_{L^2}^{2}
=\int_{\mathbb{R}^{2}}[\Lambda^{s},\,u\cdot\nabla]\omega \Lambda^{s}\omega\,dx +\int_{\mathbb{R}^{2}}\Lambda^{s}\nabla^{\perp}\nabla\cdot(b\otimes b)\Lambda^{s}\omega.
\end{align*}
Applying $\Lambda^{s+2-\alpha'}$ to $\eqref{logRG2DMHD}_{2}$ and multiplying it by $\Lambda^{s+2-\alpha'}b$ with $\alpha'<\alpha$, one gets
\begin{align*}
\frac{1}{2}\frac{d}{dt}\|\Lambda^{s+2-\alpha'}b(t)\|_{L^2}^{2}
&=\int_{\mathbb{R}^{2}}{\Lambda^{s+2-\alpha'}(b\cdot \nabla u) \cdot \Lambda^{s+2-\alpha'}b\,dx} \nonumber\\&-\int_{\mathbb{R}^{2}}{[\Lambda^{s+2-\alpha'},u\cdot \nabla] b \cdot \Lambda^{s+2-\alpha'}b\,dx}.
\end{align*}
Thanks to \eqref{logtt002}, we get that
\begin{align}\label{fzubvc3erf}
\|\Lambda^{s+\alpha'}\omega\|_{L^2}\leq C(\|\omega\|_{L^{2}}+\|\Lambda^{s+\alpha}
\mathcal{L}\omega\|_{L^2}),\quad \mbox{for}\ \ 0<\alpha'<\alpha.
\end{align}
Due to $\nabla\cdot u=0$ and \eqref{yfzdf68b}, it implies
\begin{align*}  &
\int_{\mathbb{R}^{2}}[\Lambda^{s},\,u\cdot\nabla]\omega \Lambda^{s}\omega\,dx \nonumber\\&=
\int_{\mathbb{R}^{2}}[\Lambda^{s}\partial_{i},\,u_{i}]\omega \Lambda^{s}\omega\,dx\nonumber\\& \leq C\|[\Lambda^{s}\partial_{i},\,u_{i}]\omega \|_{L^{2}}\|\Lambda^{s}\omega\|_{L^{2}}\nonumber\\& \leq C(\|\nabla u\|_{L^{\infty}}\|\Lambda^{s}\omega\|_{L^{2}}+\|\omega\|_{L^{p_{2}}}\|\Lambda^{s+1}
u\|_{L^{\frac{2p_{2}}{p_{2}-2}}})\|\Lambda^{s}\omega\|_{L^{2}}
\nonumber\\&
\leq C\left(\|\nabla u\|_{L^{\infty}}\|\Lambda^{s}\omega\|_{L^{2}}+
(\|\omega\|_{L^{2}}+\|\Lambda^{\alpha}\mathcal{L}\omega\|_{L^{2}})
\|\Lambda^{s+\alpha}\mathcal{L}\omega\|_{L^{2}}\right)\|\Lambda^{s}\omega\|_{L^{2}}
\nonumber\\& \leq \frac{1}{8}\|\Lambda^{s+\alpha}\mathcal{L}\omega\|_{L^2}^{2}+C(\|\nabla u\|_{L^{\infty}}+\|\omega\|_{L^{2}}^{2}+\|\Lambda^{\alpha}\mathcal{L}\omega\|_{L^{2}}
^{2})\|\Lambda^{s}\omega\|_{L^{2}}^{2},
\end{align*}
where $p_{2}>2$ satisfies
$$\frac{1-\alpha}{2}<\frac{1}{p_{2}}<\frac{\alpha+2\gamma}{2}.$$
By means of \eqref{yzzqq1} and \eqref{fzubvc3erf}, we get
\begin{align*}
\int_{\mathbb{R}^{2}}\Lambda^{s}\nabla^{\perp}\nabla\cdot(b\otimes b)\Lambda^{s}\omega&\leq C\|\Lambda^{s-\alpha'}\nabla^{\perp}\nabla\cdot(b\otimes b)\|_{L^{2}}\|\Lambda^{s+\alpha'}\omega\|_{L^{2}}\nonumber\\&\leq C\|b\|_{L^{\infty}}\|\Lambda^{s+2-\alpha'}b\|_{L^{2}}\|\Lambda^{s+\alpha'}
\omega\|_{L^{2}}
\nonumber\\&\leq C\|b\|_{L^{\infty}}\|\Lambda^{s+2-\alpha'}b\|_{L^{2}}(\|\omega\|_{L^{2}}+\|\Lambda^{s+\alpha}
\mathcal{L}\omega\|_{L^2})
\nonumber\\&\leq\frac{1}{8}\|\Lambda^{s+\alpha}\mathcal{L}\omega\|_{L^{2}}^{2}+ C\|b\|_{L^{\infty}}^{2} \|\Lambda^{s+2-\alpha'}b\|_{L^{2}}^{2}\nonumber\\&\quad +
C\|b\|_{L^{\infty}}\|\omega\|_{L^{2}}\|\Lambda^{s+2-\alpha'}b\|_{L^{2}}.
\end{align*}
Similarly, we can verify
\begin{align*} &
\int_{\mathbb{R}^{2}}{\Lambda^{s+2-\alpha'}(b\cdot \nabla u) \cdot \Lambda^{s+2-\alpha'}b\,dx}\nonumber\\&\leq C(\|b\|_{L^{\infty}} \|\Lambda^{s+3-\alpha'}
u\|_{L^{2}}+\|\nabla u\|_{L^{\infty}} \|\Lambda^{s+2-\alpha'}
b\|_{L^{2}})\|\Lambda^{s+2-\alpha'}b\|_{L^{2}}
\nonumber\\&\leq C(\|b\|_{L^{\infty}} \|\Lambda^{s+\alpha'}
\omega\|_{L^{2}}+\|\nabla u\|_{L^{\infty}} \|\Lambda^{s+2-\alpha'}
b\|_{L^{2}})\|\Lambda^{s+2-\alpha'}b\|_{L^{2}}
\nonumber\\&\leq C(\|b\|_{L^{\infty}} \|\omega\|_{L^{2}}+\|b\|_{L^{\infty}} \|\Lambda^{s+\alpha}\mathcal{L}\omega\|_{L^{2}}+\|\nabla u\|_{L^{\infty}} \|\Lambda^{s+2-\alpha'}
b\|_{L^{2}})\|\Lambda^{s+2-\alpha'}b\|_{L^{2}}
\nonumber\\&\leq\frac{1}{8} \|\Lambda^{s+\alpha}\mathcal{L}\omega\|_{L^{2}}^{2}+ C(\|\nabla u\|_{L^{\infty}}+\|b\|_{L^{\infty}}^{2})\|\Lambda^{s+2-\alpha'}b\|_{L^{2}}^{2}
\nonumber\\&\quad+C
\|b\|_{L^{\infty}}\|\omega\|_{L^{2}}\|\Lambda^{s+2-\alpha'}b\|_{L^{2}}.
\end{align*}
Using $\nabla\cdot u=0$ and \eqref{yfzdf68b}, it yields
\begin{align*} &
\int_{\mathbb{R}^{2}}{\Lambda^{s+2-\alpha'}(u\cdot \nabla b) \cdot \Lambda^{s+2-\alpha'}b\,dx}\nonumber\\&=\int_{\mathbb{R}^{2}}{
[\Lambda^{s+2-\alpha'}\partial_{x_{i}},
u_{i}]b_{j}  \Lambda^{s+2-\alpha'}b_{j}\,dx}
\nonumber\\&\leq C\|[\Lambda^{s+2-\alpha'}\partial_{x_{i}},
u_{i}]b_{j} \|_{L^{2}}\|\Lambda^{s+2-\alpha'}b\|_{L^{2}}
\nonumber\\&\leq
C(\|b\|_{L^{\infty}} \|\Lambda^{s+3-\alpha'}
u\|_{L^{2}}+\|\nabla u\|_{L^{\infty}} \|\Lambda^{s+2-\alpha'}
b\|_{L^{2}})\|\Lambda^{s+2-\alpha'}b\|_{L^{2}}
\nonumber\\&\leq C(\|b\|_{L^{\infty}} \|\Lambda^{s+\alpha'}
\omega\|_{L^{2}}+\|\nabla u\|_{L^{\infty}} \|\Lambda^{s+2-\alpha'}
b\|_{L^{2}})\|\Lambda^{s+2-\alpha'}b\|_{L^{2}}
\nonumber\\&\leq\frac{1}{8} \|\Lambda^{s+\alpha}\mathcal{L}\omega\|_{L^{2}}^{2}+ C(\|\nabla u\|_{L^{\infty}}+\|b\|_{L^{\infty}}^{2})\|\Lambda^{s+2-\alpha'}b\|_{L^{2}}^{2}\nonumber\\&\quad+C
\|b\|_{L^{\infty}}\|\omega\|_{L^{2}}\|\Lambda^{s+2-\alpha'}b\|_{L^{2}}.
\end{align*}
Collecting all the above estimates, one derives
\begin{align*} &
\frac{d}{dt}(\|\Lambda^{s}\omega(t)\|_{L^2}^{2}+\|\Lambda^{s+2-\alpha'}b(t)\|_{L^2}^{2})
+\|\Lambda^{s+\alpha}\mathcal{L}
\omega\|_{L^2}^{2}\nonumber\\&\leq C(\|\nabla u\|_{L^{\infty}}+\|\omega\|_{L^{2}}^{2}+\|\Lambda^{\alpha}\mathcal{L}\omega\|_{L^{2}}
^{2}+\|b\|_{L^{\infty}}^{2})
(\|\Lambda^{s}\omega\|_{L^2}^{2}+\|\Lambda^{s+2-\alpha'}b\|_{L^2}^{2})
\nonumber\\&\quad+C
\|b\|_{L^{\infty}}\|\omega\|_{L^{2}}\|\Lambda^{s+2-\alpha'}b\|_{L^{2}}.
\end{align*}
By \eqref{RMHDyet302}, \eqref{sdfrD21rd} and \eqref{Rfdty78i17}, we have
$$\int_{0}^{T}(\|\nabla u\|_{L^{\infty}}+\|\omega\|_{L^{2}}^{2}+\|\Lambda^{\alpha}\mathcal{L}\omega\|_{L^{2}}
^{2}+\|b\|_{L^{\infty}}^{2}+\|b\|_{L^{\infty}}\|\omega\|_{L^{2}})(t)\,dt\leq C(T,v_{0},b_{0}),$$
which together with the Gronwall inequality entails
$$\sup_{t\in[0,\,T]}(\|\Lambda^{s}\omega(t)\|_{L^2}^{2}+\|\Lambda^{s+2-\alpha'}b(t)\|_{L^2}^{2})
+\int_{0}^{T}\|\Lambda^{s+\alpha}\mathcal{L}
\omega(\tau)\|_{L^2}^{2}\,d\tau\leq C(T,v_{0},b_{0}).$$
Consequently, this completes the proof of Theorem \ref{Th2}.
\end{proof}

\vskip .2in

\begin{center}
The proof of Theorem \ref{addTh2}
\end{center}
\vskip .1in
The proof of Theorem \ref{addTh2} can be performed as that of Theorem \ref{Th2}.
For the sake of convenience, we sketch its proof as follows. First, we have
\begin{lemma}\label{yelem620l1}
Assume $(v_{0},\,b_{0})$ satisfies the assumptions stated in
Theorem \ref{addTh2}.
Then the corresponding solution $(v, b)$
of (\ref{endlogRG2DMHD}) admits the following bound
\begin{align}\label{yelem620t01}
\|u(t)\|_{L^{2}}^{2}+\Big\|\frac{\Lambda^{2}}{g(\Lambda)} u(t)\Big\|_{L^{2}}^{2}+\|b(t)\|_{L^{2}}^{2}
\leq
C(v_{0},b_{0}).
\end{align}
\end{lemma}

\vskip .1in
Modifying the proof of Lemma \ref{RMHDS32}, we will derive the following crucial estimate.

\begin{lemma}\label{yelem620l2}
Assume $(v_{0},\,b_{0})$ satisfies the assumptions stated in
Theorem \ref{addTh2}. Then the corresponding solution $(v, b)$
of (\ref{endlogRG2DMHD}) admits the following bound
\begin{align}\label{yelem620t02}
\sup_{t\in [0,\,T]}(\|b(t)\|_{L^{\infty}}
+\|\nabla b(t)\|_{L^{2}}
+\|v(t)\|_{L^{2}}+\|\nabla b(t)\|_{L^{\infty}})\leq C(T,v_{0},b_{0}).
\end{align}
Due to $v=u+(-\Delta)^{2}\mathcal{L}^{2}u$, it holds true
\begin{align}\label{yelem620t03}
\sup_{t\in [0,\,T]}\|\Lambda^{4} \mathcal{L}^{2}u(t)\|_{L^{2}}\leq
C(T,v_{0},b_{0}).
\end{align}
\end{lemma}

\begin{proof}
Following the arguments in deriving \eqref{RMHDzxr21}, it is not difficult to check that
\begin{align}\label{yelem620t0324}
\frac{d}{dt}V(t)\leq C(1+\|\nabla u\|_{L^{\infty}}+\|\Delta u\|_{L^{2}})V(t),
\end{align}
where
$$V(t):=\|b(t)\|_{L^{\infty}}^{3}+\|b(t)\|_{L^{\infty}}^{6}+
\|\nabla b(t)\|_{L^{2}}^{3}+\|v(t)\|_{L^{2}}^{2}.$$
Moreover, one may also conclude that for some $r_{1}>2$
\begin{eqnarray}
\|\nabla u\|_{L^{\infty}}\leq C\|u\|_{L^{2}}+Cg (2^{N})\sqrt{N} \Big\|\frac{\Lambda^{2}}{ g (\Lambda)}u \Big\|_{L^{2}}+C2^{N(2-r_{1})}\|\Lambda^{r_{1}}u\|_{L^{2}},\nonumber
\end{eqnarray}
\begin{eqnarray}
\|\Delta u\|_{L^{2}}\leq C\|u\|_{L^{2}}+Cg(2^{N})\sqrt{N} \Big\|\frac{\Lambda^{2}}{ g(\Lambda)}u \Big\|_{L^{2}}+C2^{N(2-r_{1})}\|\Lambda^{r_{1}}u\|_{L^{2}}.\nonumber
\end{eqnarray}
Making use of the following fact due to $v=u+(-\Delta)^{2}\mathcal{L}^{2}u$
$$C2^{N(2-r_{1})}\|\Lambda^{r_{1}}u\|_{L^{2}}\leq C(\|u\|_{L^{2}}+2^{N(2-r_{1})}\|v\|_{L^{2}}),$$
it yields that
\begin{eqnarray}
 \frac{d}{dt}V(t)
 \leq    CV(t)+Cg(2^{N})\sqrt{N} \Big\|\frac{\Lambda^{2}}{ g(\Lambda)}u \Big\|_{L^{2}}V(t)+C2^{N(2-r_{1})}V^{\frac{1}{2}}(t)V(t).\nonumber
\end{eqnarray}
Taking $N$ such
$$2^{N}\approx \left(e+V(t)\right)^{\frac{1}{2(r_{1}-2)}},$$
we have
\begin{align}
\frac{d}{dt}V(t)
\leq    C\left(1+\Big\|\frac{\Lambda^{2}}{ g (\Lambda)}u \Big\|_{L^{2}}\right)g \big[\left(e+V(t)\right)^{\frac{1}{2(r_{1}-2)}}\big]\sqrt{\ln\big(e+A(t)\big)} \big(e+V(t)\big).\nonumber
\end{align}
It thus follows from the above inequality that
$$\int_{e+V(0)}^{e+V(t)}\frac{d\tau}{\tau\sqrt{\ln \tau} g (\tau^{\frac{1}{2(r_{1}-2)}})}\leq C
\int_{0}^{t}{
\left(1+\Big\|\frac{\Lambda^{2}}{ g (\Lambda)}u (\tau) \Big\|_{L^{2}}\right)\,d\tau}.$$
By \eqref{endlogcobd}, we have
$$\int_{e}^{\infty}\frac{d\tau}{\tau\sqrt{\ln \tau} g(\tau^{\frac{1}{2(r_{1}-2)}})}=\sqrt{2(r_{1}-2)}\int_{e^{\frac{1}{2(r_{1}-2)}}}^{\infty}\frac{d\tau}{\tau\sqrt{\ln \tau} g(\tau)}=\infty.$$
Therefore, noticing the following fact due to \eqref{yelem620t01}
$$\int_{0}^{T}{
\left(1+\Big\|\frac{\Lambda^{2}}{ g(\Lambda)}u (\tau) \Big\|_{L^{2}}\right)\,d\tau}\leq C(T,\,u_{0},\,w_{0}),$$
we can verify that
\begin{align}\label{yelem620t0326}
\sup_{t\in [0,\,T]}V(t)\leq C(T,v_{0},b_{0}).
\end{align}
Coming back to \eqref{ertcxp316}, we can show
\begin{align*}
\frac{d}{dt}\|\nabla b(t)\|_{L^{\infty}} &\leq C\|\nabla u\|_{L^{\infty}}\|\nabla b\|_{L^{\infty}} +C\|b\|_{L^{\infty}}\|\nabla^{2} u\|_{L^{\infty}}\|\nabla b\|_{L^{\infty}} \nonumber\\&\leq C(\|u\|_{L^{2}}+\|\Lambda^{4} \mathcal{L}^{2}u(t)\|_{L^{2}})\|\nabla b\|_{L^{\infty}}\nonumber\\&\quad+C\|b\|_{L^{\infty}}
(\|u\|_{L^{2}}+\|\Lambda^{4} \mathcal{L}^{2}u(t)\|_{L^{2}})\|\nabla b\|_{L^{\infty}},
\end{align*}
which along with the Gronwall inequality and \eqref{yelem620t0326} gives
$$\sup_{t\in [0,\,T]}\|\nabla b(t)\|_{L^{\infty}}\leq C(T,v_{0},b_{0}).$$
This concludes the proof of Lemma \ref{yelem620l2}.
\end{proof}

\vskip .1in
The following lemma concerns the $L^q$-estimate of $\omega$ for some $q>2$.
\begin{lemma}\label{yelem620l3}
Assume $(v_{0},\,b_{0})$ satisfies the assumptions stated in
Theorem \ref{addTh2}. Then the corresponding solution $(v, b)$
of (\ref{endlogRG2DMHD}) admits the following bound for some $q>2$
\begin{align}\label{yelem620t0327}
\sup_{t\in [0,\,T]}\|\omega(t)\|_{L^{q}}\leq
C(T,v_{0},b_{0}).
\end{align}
\end{lemma}

\begin{proof}
For some $0<\theta<1$, we claim
\begin{align}\label{yelem620t0328}
\sup_{t\in [0,\,T]}\|\Lambda^{2+\theta} b(t)\|_{L^{2}}\leq
C(T,v_{0},b_{0}).
\end{align}
Applying $\Lambda^{2+\theta}$ to $\eqref{endlogRG2DMHD}_{2}$ and multiplying it by $\Lambda^{2+\theta}b$, we infer
\begin{align*}
\frac{1}{2}\frac{d}{dt}\|\Lambda^{2+\theta} b(t)\|_{L^{2}}^{2}&=-\int_{\mathbb{R}^{2}}{[\Lambda^{2+\theta}, u\cdot \nabla ]b \cdot \Lambda^{2+\theta} b\,dx}+\int_{\mathbb{R}^{2}}{\Lambda^{2+\theta}(b\cdot \nabla u) \cdot\Lambda^{2+\theta} b\,dx}
\nonumber\\&\leq C\|[\Lambda^{2+\theta}, u\cdot \nabla ]b\|_{L^{2}}\|\Lambda^{2+\theta} b\|_{L^{2}}+C\|\Lambda^{2+\theta}(b\cdot \nabla u)\|_{L^{2}}\|\Lambda^{2+\theta} b\|_{L^{2}}
\nonumber\\&\leq C(\|\nabla u\|_{L^{\infty}}\|\Lambda^{2+\theta} b\|_{L^{2}}+\|\nabla b\|_{L^{\infty}}\|\Lambda^{2+\theta} u\|_{L^{2}})\|\Lambda^{2+\theta} b\|_{L^{2}}\nonumber\\&\quad+C(\|\nabla u\|_{L^{\infty}}\|\Lambda^{2+\theta} b\|_{L^{2}}+\|b\|_{L^{\infty}}\|\Lambda^{3+\theta}u\|_{L^{2}})\|\Lambda^{2+\theta} b\|_{L^{2}},
\end{align*}
which allows us to show
\begin{align*}
 \frac{d}{dt}\|\Lambda^{2+\theta} b(t)\|_{L^{2}} &\leq C(\|u\|_{L^{2}}+\|\Lambda^{4} \mathcal{L}^{2}u(t)\|_{L^{2}})\|\Lambda^{2+\theta} b\|_{L^{2}}\nonumber\\&\quad+C(\|\nabla b\|_{L^{\infty}}+\|b\|_{L^{\infty}})
 (\|u\|_{L^{2}}+\|\Lambda^{4} \mathcal{L}^{2}u(t)\|_{L^{2}}).
\end{align*}
Thanks to \eqref{yelem620t02} and \eqref{yelem620t03}, we thus obtain the desired estimate \eqref{yelem620t0328}.
Let us recall the vorticity $\omega:=\nabla^{\perp}\cdot v$, which satisfies
\begin{align}\label{yelem620t0329}
\partial_{t}\omega+(u\cdot\nabla)\omega=\nabla^{\perp}\nabla\cdot(b\otimes b).\end{align}
Taking the $L^2$ inner product of \eqref{yelem620t0329} with $|\omega|^{q-2}\omega$, it implies
\begin{align*}
\frac{1}{q}\frac{d}{dt}\|\omega(t)\|_{L^q}^{q}
&\leq
C\|\nabla^{\perp}\nabla\cdot(b\otimes b)\|_{L^q}\|\omega\|_{L^q}^{q-1}\nonumber\\
&\leq
C\|b\|_{L^{\infty}}\|\Delta b\|_{L^{q}}\|\omega\|_{L^q}^{q-1}
\nonumber\\
&\leq
C\|b\|_{L^{\infty}}\|b\|_{H^{2+\theta}}\|\omega\|_{L^q}^{q-1},
\end{align*}
which leads to
$$\frac{d}{dt}\|\omega(t)\|_{L^q}\leq C\|b\|_{L^{\infty}}\|b\|_{H^{2+\theta}}.$$
Integrating it in time and using \eqref{yelem620t0328}, we have for some $q>2$
$$\sup_{t\in [0,\,T]}\|\omega(t)\|_{L^{q}} \leq
C(T,v_{0},b_{0}).$$
Consequently, Lemma \ref{yelem620l3} is proved.
\end{proof}

\vskip .1in
With the above estimates at our disposal, we will prove Theorem \ref{addTh2}.
\begin{proof}[{Proof of Theorem \ref{addTh2}}]
Applying $\Lambda^{s}$ to $\eqref{yelem620t0329}$ and multiplying it by $\Lambda^{s}\omega$ yield
\begin{align*}
\frac{1}{2}\frac{d}{dt}\|\Lambda^{s}\omega(t)\|_{L^2}^{2}
=\int_{\mathbb{R}^{2}}[\Lambda^{s},\,u\cdot\nabla]\omega \Lambda^{s}\omega\,dx +\int_{\mathbb{R}^{2}}\Lambda^{s}\nabla^{\perp}\nabla\cdot(b\otimes b)\Lambda^{s}\omega.
\end{align*}
Applying $\Lambda^{s+2}$ to $\eqref{endlogRG2DMHD}_{2}$ and multiplying it by $\Lambda^{s+2}b$, we infer
\begin{align*}
\frac{1}{2}\frac{d}{dt}\|\Lambda^{s+2}b(t)\|_{L^2}^{2}
=\int_{\mathbb{R}^{2}}{\Lambda^{s+2}(b\cdot \nabla u) \cdot \Lambda^{s+2}b\,dx} -\int_{\mathbb{R}^{2}}{[\Lambda^{s+2},u\cdot \nabla] b \cdot \Lambda^{s+2}b\,dx}.
\end{align*}
Using the same arguments used in proving Theorem \ref{Th2}, we may check
\begin{align*}
\int_{\mathbb{R}^{2}}[\Lambda^{s},\,u\cdot\nabla]\omega \Lambda^{s}\omega\,dx & \leq C\|[\Lambda^{s}\partial_{i},\,u_{i}]\omega \|_{L^{2}}\|\Lambda^{s}\omega\|_{L^{2}}\nonumber\\& \leq C(\|\nabla u\|_{L^{\infty}}\|\Lambda^{s}\omega\|_{L^{2}}+\|\omega\|_{L^{p}}\|\Lambda^{s+1}
u\|_{L^{\frac{2p}{p-2}}})\|\Lambda^{s}\omega\|_{L^{2}}
\nonumber\\&
\leq C\|\nabla u\|_{L^{\infty}}\|\Lambda^{s}\omega\|_{L^{2}}^{2}+C\|\omega\|_{L^{p}}
\|\Lambda^{s}\omega\|_{L^{2}}^{2},
\end{align*}
\begin{align*}
\int_{\mathbb{R}^{2}}\Lambda^{s}\nabla^{\perp}\nabla\cdot(b\otimes b)\Lambda^{s}\omega&\leq C\|\Lambda^{s}\nabla^{\perp}\nabla\cdot(b\otimes b)\|_{L^{2}}\|\Lambda^{s}\omega\|_{L^{2}}\nonumber\\&\leq C\|b\|_{L^{\infty}}\|\Lambda^{s+2}b\|_{L^{2}}\|\Lambda^{s}
\omega\|_{L^{2}},
\end{align*}
\begin{align*}
\int_{\mathbb{R}^{2}}{\Lambda^{s+2}(b\cdot \nabla u) \cdot \Lambda^{s+2}b\,dx} &\leq C(\|b\|_{L^{\infty}} \|\Lambda^{s+3}
u\|_{L^{2}}+\|\nabla u\|_{L^{\infty}} \|\Lambda^{s+2}
b\|_{L^{2}})\|\Lambda^{s+2}b\|_{L^{2}}
\nonumber\\&\leq C(\|b\|_{L^{\infty}} \|\Lambda^{s}
\omega\|_{L^{2}}+\|\nabla u\|_{L^{\infty}} \|\Lambda^{s+2}
b\|_{L^{2}})\|\Lambda^{s+2}b\|_{L^{2}},
\end{align*}
\begin{align*}
\int_{\mathbb{R}^{2}}{\Lambda^{s+2}(u\cdot \nabla b) \cdot \Lambda^{s+2}b\,dx}
&\leq C\|[\Lambda^{s+2}\partial_{i},
u_{i}]b_{j} \|_{L^{2}}\|\Lambda^{s+2}b\|_{L^{2}}
\nonumber\\&\leq
C(\|b\|_{L^{\infty}} \|\Lambda^{s+3}
u\|_{L^{2}}+\|\nabla u\|_{L^{\infty}} \|\Lambda^{s+2}
b\|_{L^{2}})\|\Lambda^{s+2}b\|_{L^{2}}
\nonumber\\&\leq C(\|b\|_{L^{\infty}} \|\Lambda^{s}
\omega\|_{L^{2}}+\|\nabla u\|_{L^{\infty}} \|\Lambda^{s+2}
b\|_{L^{2}})\|\Lambda^{s+2}b\|_{L^{2}}.
\end{align*}
Summing up all the above estimates, it holds that
\begin{align*} &
\frac{d}{dt}(\|\Lambda^{s}\omega(t)\|_{L^2}^{2}+\|\Lambda^{s+2}b(t)\|_{L^2}^{2})
\nonumber\\&\leq C(\|\omega\|_{L^{p}}+\|\nabla u\|_{L^{\infty}}+\|b\|_{L^{\infty}})
(\|\Lambda^{s}\omega\|_{L^2}^{2}+\|\Lambda^{s+2}b\|_{L^2}^{2})
\nonumber\\&\leq C(\|\omega\|_{L^{p}}+\|u\|_{L^{2}}+\|\Lambda^{4} \mathcal{L}^{2}u\|_{L^{2}}+\|b\|_{L^{\infty}})
(\|\Lambda^{s}\omega\|_{L^2}^{2}+\|\Lambda^{s+2}b\|_{L^2}^{2}).
\end{align*}
In view of the Gronwall inequality, we obtain by using \eqref{yelem620t03} and \eqref{yelem620t0327} that
$$\sup_{t\in[0,\,T]}(\|\Lambda^{s}\omega(t)\|_{L^2}^{2}+\|\Lambda^{s+2}b(t)\|_{L^2}^{2})
\leq C(T,v_{0},b_{0}).$$
As a result, we finish the proof of Theorem \ref{addTh2}.
\end{proof}

\vskip .3in
\appendix

\section{The proof of Theorem \ref{Th3}}
\label{aaapA}

To begin with, the basic energy estimate reads as follows.
\begin{lemma}\label{AppaAl01}
Assume $(v_{0},\,b_{0})$ satisfies the assumptions stated in
Theorem \ref{Th3}.
Then the corresponding solution $(v, b)$
of (\ref{addRG2DMHD}) admits the following bound for any $t\in[0,\,T]$
\begin{align}\label{AppaAye01}
\|u(t)\|_{L^{2}}^{2}+\|\nabla u(t)\|_{L^{2}}^{2}
+\|b(t)\|_{L^{2}}^{2}
+\int_{0}^{t}{\|\Lambda^{\beta}b(\tau)\|_{L^{2}}^{2}\,d\tau}\leq
C(v_{0},b_{0}).
\end{align}
\end{lemma}

\begin{proof}
Taking the inner product of $(\ref{addRG2DMHD})_{1}$ with $u$ and
the inner product of $(\ref{addRG2DMHD})_{2}$ with $b$, one gets
\begin{align*}
\frac{1}{2}\frac{d}{dt}(\|u(t)\|_{L^{2}}^{2}+\|\nabla u(t)\|_{L^{2}}^{2}
+\|b(t)\|_{L^{2}}^{2})+\|\Lambda^{\beta} b\|_{L^{2}}^{2}=0,
\end{align*}
where we have used the cancelations \eqref{cancc01}-\eqref{cancc02} and the following crucial identity
$$\int_{\mathbb{R}^{2}}{(u\cdot \nabla v)\cdot
 u \,dx}=0.$$
Actually, the above equality can be deduced as
\begin{align*}
\int_{\mathbb{R}^{2}}{(u\cdot \nabla v)\cdot
 u \,dx}&=\int_{\mathbb{R}^{2}}{(u\cdot \nabla u)\cdot
 u \,dx}-\int_{\mathbb{R}^{2}}{(u\cdot\nabla\Delta u)\cdot
 u \,dx}\nonumber\\&=-\int_{\mathbb{R}^{2}}{(u\cdot\nabla \Delta u)\cdot
 u \,dx}\nonumber\\&=-\int_{\mathbb{R}^{2}}{u_{i}\partial_{i} \partial_{l}^{2} u_{j}u_{j} \,dx}\nonumber\\&=\int_{\mathbb{R}^{2}}{u_{i} \partial_{l}^{2} u_{j}\partial_{i}u_{j} \,dx}
 \nonumber\\&=\int_{\mathbb{R}^{2}}{(u\cdot \nabla u)\cdot
\Delta u \,dx} \nonumber\\&=0,
\end{align*}
where in the last line we have used the following crucial fact
due to $\nabla\cdot u=0$ (see \cite[(3.3)]{Wuxye} for details)
$$\int_{\mathbb{R}^{2}}{(u\cdot \nabla u)\cdot
\Delta u \,dx}=0.$$
Integrating in time yields the desired (\ref{AppaAye01}).
\end{proof}

\vskip .1in
According to the proof of Lemma \ref{RMHDY2} (letting $\gamma=1$), the following lemma holds immediately.
\begin{lemma}\label{AppaAl02}
Assume $(v_{0},\,b_{0})$ satisfies the assumptions stated in
Theorem \ref{Th3}. If $\beta>\frac{1}{2}$, then the corresponding solution $(v, b)$ of (\ref{addRG2DMHD}) admits the following bound
\begin{align}\label{AppaAye02}
\sup_{t\in [0,\,T]}\|\Lambda^{\beta}b(t)\|_{L^{2}}^{2}
+\int_{0}^{T}{\|\Lambda^{2\beta}b(\tau)\|_{L^{2}}^{2}\,d\tau}\leq
C(T,v_{0},b_{0}).
\end{align}
\end{lemma}

\vskip .1in
In this case, the vorticity $\omega:=\nabla^{\perp}\cdot v$ satisfies
\begin{align}\label{AppaAye04}
\partial_{t}\omega+(u\cdot\nabla)\omega=\nabla^{\perp}\nabla\cdot(b\otimes b)+\nabla^{\perp}u_{i}\partial_{i}v.
\end{align}
which is different from \eqref{RMHDTQye07}. Due to the presence of $\nabla^{\perp}u_{i}\partial_{i}v$ in the vorticity equation \eqref{AppaAye04}, the following lemma cannot be obtained as that of Lemma \ref{RMHDY3}.
\begin{lemma}\label{AppaAl03}
Assume $(v_{0},\,b_{0})$ satisfies the assumptions stated in
Theorem \ref{Th3}. If $\beta>\frac{1}{2}$, then the corresponding solution $(v, b)$ of (\ref{addRG2DMHD}) admits the following bound
\begin{align}\label{AppaAye05}
\sup_{t\in [0,\,T]}(\|v(t)\|_{H^{1}}^{2}+\|\Lambda^{2+\beta}b(t)\|_{L^{2}}^{2})
+\int_{0}^{T}{\|\Lambda^{2+2\beta}b(\tau)\|_{L^{2}}^{2}\,d\tau}\leq
C(T,v_{0},b_{0}).
\end{align}
\end{lemma}
\begin{proof}
First, taking the $L^2$-inner product of $\eqref{addRG2DMHD}_{1}$ with $v$, it gives
\begin{align}\label{AppaAyf66}
\frac{1}{2}\frac{d}{dt}\|v(t)\|_{L^{2}}^{2}&=\int_{\mathbb{R}^{2}}{(b\cdot \nabla b) \cdot v\,dx}\nonumber\\&\leq C\| b\|_{L^{\infty}}\|\nabla b\|_{L^{2}}\|v\|_{L^{2}}\nonumber\\&\leq C\|b\|_{H^{2\beta}}^{2}\|v\|_{L^{2}},
\end{align}
where here and in what follows we have used
$$\| b\|_{L^{\infty}}\leq C\|b\|_{H^{2\beta}}, \quad \beta>\frac{1}{2}.$$
Thanks to \eqref{AppaAye02}, we get from \eqref{AppaAyf66} that
\begin{align}\label{AppaAye07}
\sup_{t\in [0,\,T]}\|v(t)\|_{L^{2}}\leq
C(T,v_{0},b_{0}).
\end{align}
Keeping in mind the relation $v=u-\Delta u$, we obtain
\begin{align}\label{AppaAye08}
\sup_{t\in [0,\,T]}\|u(t)\|_{H^{2}}\leq
C(T,v_{0},b_{0}).
\end{align}
Taking the $L^2$-inner product of \eqref{AppaAye04} with $\omega$ yields
\begin{align*}
\frac{1}{2}\frac{d}{dt}\|\omega(t)\|_{L^{2}}^{2}=\int_{\mathbb{R}^{2}}
{\nabla^{\perp}\nabla\cdot(b\otimes b)\,\omega\,dx}+\int_{\mathbb{R}^{2}}{\nabla^{\perp}u_{i}\partial_{i}v\,\omega\,dx}.
\end{align*}
According to \eqref{RMHDTQye08}, one has
\begin{align*}
\int_{\mathbb{R}^{2}}
{\nabla^{\perp}\nabla\cdot(b\otimes b)\,\omega\,dx}\leq \frac{1}{4}\|\Lambda^{2+2\beta}b\|_{L^{2}}^{2}+C(1+\| b\|_{L^{\infty}}^{2})(1+\|\omega\|_{L^{2}}^{2}).
\end{align*}
Moreover, it is obvious that
$$\int_{\mathbb{R}^{2}}{\nabla^{\perp}u_{i}\partial_{i}v\,\omega\,dx}\leq C\|\nabla u\|_{L^{\infty}}\|\nabla v\|_{L^{2}}\|\omega\|_{L^{2}}\leq C\|\nabla u\|_{L^{\infty}}\|\omega\|_{L^{2}}^{2}.$$
Whence, we derive
\begin{align}\label{AppaAye09}
\frac{1}{2}\frac{d}{dt}\|\omega(t)\|_{L^{2}}^{2}\leq \frac{1}{4}\|\Lambda^{2+2\beta}b\|_{L^{2}}^{2}+C(1+\|\nabla u\|_{L^{\infty}}+\|b\|_{L^{\infty}}^{2})(1+\|\omega\|_{L^{2}}^{2}).
\end{align}
According to \eqref{cdftade117}, we deduce
\begin{align}\label{AppaAye10}
\frac{1}{2}\frac{d}{dt}\|\Lambda^{2+\beta}b(t)\|_{L^{2}}^{2}+
\frac{3}{4}\|\Lambda^{2+2\beta}b\|_{L^{2}}^{2}\leq C\|b\|_{L^{\infty}}^{2}\|\omega\|_{L^{2}}^{2}
+C\|u\|_{H^{1}}^{\frac{4(\beta+1)p_{0}}{(2\beta-1)p_{0}-2}}
\|b\|_{L^{2}}^{2},
\end{align}
where $p_{0}>\frac{2}{2\beta-1}$. Summing up \eqref{AppaAye09} and \eqref{AppaAye10} yields
\begin{align*}
\frac{d}{dt}(\|\omega(t)\|_{L^{2}}^{2}+\|\Lambda^{2+\beta}b(t)\|_{L^{2}}^{2})+
\|\Lambda^{2+2\beta}b\|_{L^{2}}^{2}\leq C(1+\|\nabla u\|_{L^{\infty}}+\|b\|_{H^{2\beta}}^{2})(1+\|\omega\|_{L^{2}}^{2}).
\end{align*}
Recalling the logarithmic Sobolev inequality (see
also \cite{KOT})
\begin{align*}
\|\nabla u\|_{L^{\infty}(\mathbb{R}^{2})}
\leq C\Big(1+\|u\|_{L^{2}(\mathbb{R}^{2})}+ \|\Delta u\|_{L^{2}(\mathbb{R}^{2})} \ln\big(e+\|
\Lambda^{\varrho}u\|_{L^{2}(\mathbb{R}^{2})}\big)\Big),\quad \varrho>2,
\end{align*}
we have by taking $\varrho=3$ that
\begin{align*}&
\frac{d}{dt}(\|\omega(t)\|_{L^{2}}^{2}+\|\Lambda^{2+\beta}b(t)\|_{L^{2}}^{2})+
\|\Lambda^{2+2\beta}b\|_{L^{2}}^{2}\nonumber\\&\leq C(1+\|\Delta u\|_{L^{2}}+\|b\|_{H^{2\beta}}^{2})\ln\big(e+\|\Lambda^{3}u
\|_{L^{2}}^{2}\big)(1+\|\omega\|_{L^{2}}^{2})\nonumber\\&\leq C(1+\|\Delta u\|_{L^{2}}+\|b\|_{H^{2\beta}}^{2})\ln\big(e+\|\omega\|_{L^{2}}^{2}
+\|\Lambda^{2+\beta}b\|_{L^{2}}^{2}\big)(1+\|\omega
\|_{L^{2}}^{2}+\|\Lambda^{2+\beta}b
\|_{L^{2}}^{2}).
\end{align*}
Making use of \eqref{AppaAye02}, \eqref{AppaAye08} and the Gronwall type inequality, we deduce that
\begin{align}\label{AppaAye11}
\sup_{t\in [0,\,T]}(\|\omega(t)\|_{L^{2}}^{2}+\|\Lambda^{2+\beta}b(t)\|_{L^{2}}^{2})
+\int_{0}^{T}{\|\Lambda^{2+2\beta}b(\tau)\|_{L^{2}}^{2}\,d\tau}\leq
C(T,v_{0},b_{0}).
\end{align}
Combining \eqref{AppaAye07} and \eqref{AppaAye11} leads to \eqref{AppaAye05}.
Therefore, we conclude the proof of Lemma \ref{AppaAl03}.
\end{proof}

\vskip .1in
With the above estimates in hand, we are now ready to prove Theorem \ref{Th3}.
\begin{proof}[{Proof of Theorem \ref{Th3}}]
Applying $(\Lambda^{\rho},\,\Lambda^{\rho+1-\beta})$ to $(\eqref{addRG2DMHD}_{1},\,\eqref{addRG2DMHD}_{2})$ and multiplying them by  $(\Lambda^{\rho}v,\,\Lambda^{\rho+1-\beta}b)$, respectively, we infer
\begin{align}\label{AppaAye12}
\frac{1}{2}\frac{d}{dt}(\|\Lambda^{\rho}v(t)\|_{L^2}^{2}+
\|\Lambda^{\rho+1-\beta}b(t)\|_{L^2}^{2})
+\|\Lambda^{\rho+1}b\|_{L^2}^{2}
=\sum_{k=1}^{4}L_{k},
\end{align}
where
$$L_{1}:=-\int_{\mathbb{R}^{2}}[\Lambda^{\rho},\,u\cdot\nabla]v\cdot \Lambda^{\rho}v\,dx,\qquad L_{2}:=\int_{\mathbb{R}^{2}}\Lambda^{\rho}\nabla\cdot(b\otimes b)\Lambda^{s}\cdot \Lambda^{\rho}v,$$
$$L_{3}:=\int_{\mathbb{R}^{2}}\Lambda^{\rho+1-\beta}(b\cdot\nabla u)\cdot \Lambda^{\rho+1-\beta}b\,dx,\qquad L_{4}:=-\int_{\mathbb{R}^{2}}\Lambda^{\rho+1-\beta}(u\cdot\nabla b)\cdot \Lambda^{\rho+1-\beta}b\,dx.$$
According to $\nabla\cdot u=0$ and \eqref{yfzdf68b}, one obtains
\begin{align} \label{AppaAye13}
L_{1} &=\int_{\mathbb{R}^{2}}[\Lambda^{\rho}\partial_{i},\,u_{i}]v_{j} \Lambda^{\rho}v_{j}\,dx\nonumber\\& \leq C\|[\Lambda^{\rho}\partial_{i},\,u_{i}]v_{j}\|_{L^{2}}\|\Lambda^{\rho}v_{j}\|_{L^{2}}\nonumber\\& \leq C(\|\nabla u\|_{L^{\infty}}\|\Lambda^{\rho}v\|_{L^{2}}
+\|v\|_{L^{4}}\|\Lambda^{\rho+1}
u\|_{L^{4}})\|\Lambda^{\rho}v\|_{L^{2}}
\nonumber\\& \leq
C(\|v\|_{H^{1}}\|\Lambda^{\rho}v\|_{L^{2}}
+\|v\|_{H^{1}}\|\Lambda^{\rho+\frac{3}{2}}
u\|_{L^{2}})\|\Lambda^{\rho}v\|_{L^{2}}
\nonumber\\& \leq C\|v\|_{H^{1}}\|\Lambda^{\rho}v\|_{L^{2}}^{2},
\end{align}
where we have used the following facts due to $v=u-\Delta u$ (see \eqref{Rcdfr6yma1} for similar proof)
$$\|\nabla u\|_{L^{\infty}}\leq C\|v\|_{H^{1}},\qquad \|\Lambda^{\rho+\frac{3}{2}}
u\|_{L^{2}}\leq C\|\Lambda^{\rho}v\|_{L^{2}}.$$
Using \eqref{yzzqq1}, it holds that
\begin{align}\label{AppaAye14}
L_{2}&\leq C\|\Lambda^{\rho}\nabla\cdot(b\otimes b)\|_{L^{2}}\|\Lambda^{\rho}v\|_{L^{2}}\nonumber\\&\leq C\|b\|_{L^{\infty}}\|\Lambda^{\rho+1}b\|_{L^{2}}\|\Lambda^{\rho}v\|_{L^{2}}
\nonumber\\&\leq C\|b\|_{H^{2+\beta}}\|\Lambda^{\rho+1}b\|_{L^{2}}\|\Lambda^{\rho}v\|_{L^{2}}
\nonumber\\&\leq\frac{1}{8}\|\Lambda^{\rho+1}b\|_{L^{2}}^{2}+ C\|b\|_{H^{2+\beta}}^{2} \|\Lambda^{\rho}v\|_{L^{2}}^{2}.
\end{align}
Due to $\nabla\cdot b=0$, we have by \eqref{yzzqq1} that
\begin{align}\label{AppaAye15}
L_{3}&=\int_{\mathbb{R}^{2}}{\Lambda^{\rho+1-\beta}\nabla\cdot
(b\otimes u) \cdot \Lambda^{\rho+1-\beta}b\,dx}
\nonumber\\&\leq C\|\Lambda^{\rho+2-2\beta}
(ub)\|_{L^{2}}\|\Lambda^{\rho+1}b\|_{L^{2}}
\nonumber\\&\leq C(\|u\|_{L^{\infty}}\|\Lambda^{\rho+2-2\beta}
b\|_{L^{2}}+\|b\|_{L^{\infty}}\|\Lambda^{\rho+2-2\beta}
u\|_{L^{2}})\|\Lambda^{\rho+1}b\|_{L^{2}}
\nonumber\\&\leq C(\|u\|_{L^{\infty}}\|b\|_{L^{2}}^{\frac{2\beta-1}{\rho+1}}\|\Lambda^{\rho+1}
b\|_{L^{2}}^{\frac{\rho+2-2\beta}{\rho+1}} +\|b\|_{L^{\infty}}\|\Lambda^{\rho}
v\|_{L^{2}})\|\Lambda^{\rho+1}b\|_{L^{2}}
\nonumber\\&\leq\frac{1}{8}\|\Lambda^{\rho+1}b\|_{L^{2}}^{2}+ C\|b\|_{L^{\infty}}^{2}\|\Lambda^{\rho}
v\|_{L^{2}}^{2}+C\|u\|_{L^{\infty}}^{\frac{2(\rho+1)}{2\beta-1}}\|b\|_{L^{2}}
^{2}\nonumber\\&\leq\frac{1}{8}\|\Lambda^{\rho+1}b\|_{L^{2}}^{2}+ C\|b\|_{H^{2+\beta}}^{2}\|\Lambda^{\rho}v\|_{L^{2}}^{2}
+C\|v\|_{H^{1}}^{\frac{2(\rho+1)}{2\beta-1}}\|b\|_{L^{2}}
^{2}.
\end{align}
Similarly, one concludes
\begin{align}\label{AppaAye16}
L_{4}&=\int_{\mathbb{R}^{2}}{\Lambda^{\rho+1-\beta}\nabla\cdot
(u\otimes b) \cdot \Lambda^{\rho+1-\beta}b\,dx}
\nonumber\\&\leq C\|\Lambda^{\rho+2-2\beta}
(ub)\|_{L^{2}}\|\Lambda^{\rho+1}b\|_{L^{2}}
\nonumber\\&\leq\frac{1}{8}\|\Lambda^{\rho+1}b\|_{L^{2}}^{2}+ C\|b\|_{H^{2+\beta}}^{2}\|\Lambda^{\rho}v\|_{L^{2}}^{2}
+C\|v\|_{H^{1}}^{\frac{2(\rho+1)}{2\beta-1}}\|b\|_{L^{2}}
^{2}.
\end{align}
Inserting \eqref{AppaAye13}, \eqref{AppaAye14}, \eqref{AppaAye15} and \eqref{AppaAye16} into \eqref{AppaAye12} yields
\begin{align}\label{AppaAye17}&
\frac{d}{dt}(\|\Lambda^{\rho}v(t)\|_{L^2}^{2}+
\|\Lambda^{\rho+1-\beta}b(t)\|_{L^2}^{2})
+\|\Lambda^{\rho+1}b\|_{L^2}^{2}\nonumber\\&\leq  C(\|v\|_{H^{1}}+\|v\|_{H^{1}}^{\frac{2(\rho+1)}{2\beta-1}}+\|b\|_{H^{2+\beta}}^{2})
(\|\Lambda^{\rho}v\|_{L^2}^{2}+
\|\Lambda^{\rho+1-\beta}b\|_{L^2}^{2}).
\end{align}
Recalling \eqref{AppaAl03}, we find
$$\int_{0}^{T}(\|v(t)\|_{H^{1}}+\|v(t)\|_{H^{1}}^{\frac{2(\rho+1)}{2\beta-1}}
+\|b(t)\|_{H^{2+\beta}}^{2})\,dt\leq C(T,v_{0},b_{0}).$$
Applying the Gronwall inequality to \eqref{AppaAye17}, we finally obtain
$$\sup_{t\in [0,\,T]}(\|\Lambda^{\rho}v(t)\|_{L^2}^{2}+
\|\Lambda^{\rho+1-\beta}b(t)\|_{L^2}^{2})
+\int_{0}^{T}\|\Lambda^{\rho+1}b(t)\|_{L^2}^{2}\,dt\leq C(T,v_{0},b_{0}).$$
Consequently, this ends the proof of Theorem \ref{Th3}.
\end{proof}

\vskip .1in
\section{Several useful facts}\setcounter{equation}{0}
\label{aaapB}

We start with the Littlewood-Paley theory. We choose
some smooth radial non increasing function $\chi$ with values in $[0, 1]$ such that $\chi\in
C_{0}^{\infty}(\mathbb{R}^{n})$ is supported in the ball
$\mathcal{B}:=\{\xi\in \mathbb{R}^{n}, |\xi|\leq \frac{4}{3}\}$ and
and with value $1$ on $\{\xi\in \mathbb{R}^{n}, |\xi|\leq \frac{3}{4}\}$, then we set
$\varphi(\xi)=\chi\big(\frac{\xi}{2}\big)-\chi(\xi)$. One easily verifies that
${\varphi\in C_{0}^{\infty}(\mathbb{R}^{n})}$ is supported in the annulus
$\mathcal{C}:=\{\xi\in \mathbb{R}^{n}, \frac{3}{4}\leq |\xi|\leq
\frac{8}{3}\}$ and satisfies
$$\chi(\xi)+\sum_{j\geq0}\varphi(2^{-j}\xi)=1, \quad  \forall \xi\in \mathbb{R}^{n}.$$
Let $h=\mathcal{F}^{-1}(\varphi)$ and $\widetilde{h}=\mathcal{F}^{-1}(\chi)$, then we introduce the dyadic blocks $\Delta_{j}$ of our decomposition by setting
$$\Delta_{j}u=0,\ \ j\leq -2; \ \  \ \ \ \Delta_{-1}u=\chi(D)u=\int_{\mathbb{R}^{n}}{\widetilde{h}(y)u(x-y)\,dy};
$$
$$ \Delta_{j}u=\varphi(2^{-j}D)u=2^{jn}\int_{\mathbb{R}^{n}}{h(2^{j}y)u(x-y)\,dy},\ \ \forall j\in \mathbb{N}.
$$
We shall also use the
following low-frequency cut-off:
$$\ S_{j}u=\chi(2^{-j}D)u=\sum_{-1\leq k\leq j-1} \Delta_{k}u=2^{jn}\int_{\mathbb{R}^{n}}{\widetilde{h}(2^{j}y)u(x-y)\,dy},\ \ \forall j\in \mathbb{N}.$$

\vskip .1in
The nonhomogeneous Besov spaces are defined through
the dyadic decomposition.
\begin{define}
Let $s\in \mathbb{R}, (p,r)\in[1,+\infty]^{2}$. The nonhomogeneous
Besov space $B_{p,r}^{s}$ is defined as a space of $f\in
S'(\mathbb{R}^{n})$ such that
$$ B_{p,r}^{s}=\{f\in S'(\mathbb{R}^{n});  \|f\|_{B_{p,r}^{s}}<\infty\},$$
where
\begin{equation}\label{1}\nonumber
 \|f\|_{B_{p,r}^{s}}=\left\{\aligned
&\Big(\sum_{j\geq-1}2^{jrs}\|\Delta_{j}f\|_{L^{p}}^{r}\Big)^{\frac{1}{r}}, \quad \forall \ r<\infty,\\
&\sup_{j\geq-1}
2^{js}\|\Delta_{j}f\|_{L^{p}}, \quad \forall \ r=\infty.\\
\endaligned\right.
\end{equation}
\end{define}

\vskip .1in
We now introduce the Bernstein's inequalities, which are useful tools in dealing with Fourier localized functions and these
inequalities trade integrability for derivatives. The following lemma provides Bernstein
type inequalities for fractional derivatives

\begin{lemma} [see \cite{BCD}]\label{vfgty8xc}
Assume $1\leq a\leq b\leq\infty$. If the integer $j\geq-1$, then it holds
$$
\|\Lambda^{k}\Delta_{j}f\|_{L^b} \le C_1\, 2^{j k  +
jn(\frac{1}{a}-\frac{1}{b})} \|\Delta_{j}f\|_{L^a},\quad k\geq 0.
$$
If the integer $j\geq0$, then we have
$$
C_2\, 2^{ j k} \|\Delta_{j}f\|_{L^b } \le \|\Lambda^{k}\Delta_{j}f\|_{L^b } \le
C_3\, 2^{  j k + j n(\frac{1}{a}-\frac{1}{b})} \|\Delta_{j}f\|_{L^a},\quad
k\in \mathbb{{R}},
$$
where $n$ is the space dimension, and $C_1$, $C_2$ and $C_3$ are constants depending on $k,a$ and $b$
only.
\end{lemma}

\vskip .1in

Let us recall the following space-time estimate, which plays a crucial role in proving Theorem \ref{Th1} (see \cite[Lemma 3.1]{Yejmaa18}).
\begin{lemma}\label{afasqw6sf}
Consider the following transport-diffusion equation with $\alpha>0$
\begin{eqnarray}
\partial_{t}f+\Lambda^{2\alpha}f=g,\qquad f(x,0)=f_{0}(x).\nonumber
\end{eqnarray}
For any $0<\widetilde{\varepsilon}\leq 2\alpha$ and for any $1\leq p,\,q\leq\infty$, let both $g$ and $e^{-\Lambda^{2\alpha} t} \,\Lambda^{2\alpha-\widetilde{\varepsilon}} f_0$ belong to $L^{q}(0,\,t;L^{p}(\mathbb{R}^{n}))$, then we have
\begin{eqnarray}
\|\Lambda^{2\alpha-\widetilde{\varepsilon}}f\|_{L_{t}^{q}L_{x}^{p}} \leq C(t,f_{0})+C(t)\|g\|_{L_{t}^{q}L_{x}^{p}}.\nonumber
\end{eqnarray}
where $C(t, f_0) = \|e^{-\Lambda^{2\alpha} t} \,\Lambda^{2\alpha-\widetilde{\varepsilon}} f_0\|_{L_{t}^{q}L_{x}^{p}}$ and $C(t)$ depends on $t$ only.
\end{lemma}

\vskip .1in
The following Kato-Ponce type commutator estimate and the product type estimate can be stated as follows (see \cite{KPonce} for example).
\begin{lemma} Let $s>0$. Assume that $p, p_1, p_3\in (1, \infty)$ and $p_2, p_4\in [1,\infty]$ satisfy
$$
\frac1p =\frac1{p_1} + \frac1{p_2} = \frac1{p_3} + \frac1{p_4}.
$$
Then there exists some constants $C$ such that
\begin{eqnarray}\label{yzz}
  \|[\Lambda^s, f]g\|_{L^p} \le C \left(\|\Lambda^s f\|_{L^{p_1}}\, \|g\|_{L^{p_2}} + \|\Lambda^{s-1} g\|_{L^{p_3}}\,\|\nabla f\|_{L^{p_4}}\right),
\end{eqnarray}
\begin{eqnarray}\label{yzzqq1}
  \|\Lambda^s(fg)\|_{L^p} \le C \left(\|\Lambda^s f\|_{L^{p_1}}\, \|g\|_{L^{p_2}} + \|\Lambda^{s} g\|_{L^{p_3}}\,\|\nabla f\|_{L^{p_4}}\right).
\end{eqnarray}
In some context, we also need the following variant version of (\ref{yzz}), whose proof is the same one as for (\ref{yzz})
\begin{eqnarray}\label{yfzdf68b}\|[\Lambda^{s-1}\partial_{i},f ]g \|_{L^r}
  \leq C\left(\|\nabla f \|_{L^{p_1}}\|\Lambda^{s-1} g \|_{L^{q_1}}
        +     \|\Lambda^{s} f \|_{L^{p_2}} \|  g \|_{L^{q_2}}\right).\end{eqnarray}
\end{lemma}

\vskip .2in

We end up this appendix with the proof of \eqref{Rcdfr6yma1}.

\textbf{The proof of \eqref{Rcdfr6yma1}}. According to the Lemma \ref{vfgty8xc}, we get that
\begin{align}
\|\nabla u\|_{L^{\infty}}&=\left\|\frac{\nabla v}{\mathbb{I}+\Lambda^{2\gamma}} \right\|_{L^{\infty}}
\nonumber\\&\leq  \left\|\Delta_{-1}\frac{\nabla v}{\mathbb{I}+\Lambda^{2\gamma}} \right\|_{L^{\infty}}+\sum_{j\geq0}\left\|\Delta_{j}\frac{\nabla v}{\mathbb{I}+\Lambda^{2\gamma}} \right\|_{L^{\infty}}
\nonumber\\&\leq C \left\|\Delta_{-1}\nabla v \right\|_{L^{\infty}}+C\sum_{j\geq0}\frac{1}{1+2^{2\gamma j}}\left\|\Delta_{j}\nabla v\right\|_{L^{\infty}}
\nonumber\\&\leq C \left\|\Delta_{-1}\nabla v \right\|_{L^{2}}+C\sum_{j\geq0}\frac{1}{1+2^{2\gamma j}}\left\|\Delta_{j}\omega \right\|_{L^{\infty}}
\nonumber\\&\leq C \left\| \omega \right\|_{L^{2}}+C\sum_{j\geq0}\frac{1}{1+2^{2\gamma j}}\left\|\omega \right\|_{L^{\infty}}
\nonumber\\&\leq C (\left\| \omega \right\|_{L^{2}}+\left\|\omega \right\|_{L^{\infty}}),\nonumber
\end{align}
where we have used $\gamma>0$ in the last line. This yields the first inequality of \eqref{Rcdfr6yma1}. The second one is the direct consequence of Plancherel theorem
\begin{align}
\|\Lambda^{s+\sigma}u\|_{L^{2}}&=\left\|\frac{\Lambda^{\sigma-1}}{\mathbb{I}+\Lambda^{2\gamma}} \Lambda^{s+1}v\right\|_{L^{2}}
\nonumber\\&=\left\|\frac{|\xi|^{\sigma-1}}{1+|\xi|^{2\gamma}} \widehat{\Lambda^{s+1}v}(\xi)\right\|_{L^{2}}
\nonumber\\&\leq C\left\| \widehat{\Lambda^{s+1}v}(\xi)\right\|_{L^{2}}
\nonumber\\&= C\left\| {\Lambda^{s+1}v} \right\|_{L^{2}}
\nonumber\\&\leq C\left\|\Lambda^{s}\omega\right\|_{L^{2}},\nonumber
\end{align}
where we have used the fact $\sigma\leq 1+2\gamma$ in the third line.

\vskip .2in
\textbf{Acknowledgements.}
The author is grateful to the referee for the invaluable comments and suggestions, which have improved the paper significantly.
The author was supported by the National Natural Science Foundation of China (No. 11701232) and the Natural Science Foundation of Jiangsu
Province (No. BK20170224). This work was carried out when the author was visiting the Department of Mathematics, University of Pittsburgh.
The author appreciates the hospitality of Professor Dehua Wang and Professor Ming Chen.

\end{document}